\documentclass[a4paper, 11pt]{article}

\usepackage[utf8]{inputenc}
\usepackage[T1]{fontenc}

\usepackage{amsmath, amsthm, amsfonts}
\usepackage{bbm}
\usepackage{enumitem}
\usepackage{graphicx}
\usepackage{ae,aecompl}
\usepackage{fullpage}
\usepackage{setspace} 
\usepackage[colorlinks=true, allcolors=blue]{hyperref}

\usepackage{wasysym}

\newtheorem{theorem}{Theorem}[section]

\newtheorem{lemma}[theorem]{Lemma}
\newtheorem{proposition}[theorem]{Proposition}

\newtheorem{conjecture}[theorem]{Conjecture}
\newtheorem*{conjecture*}{Conjecture}
\newtheorem{remark}[theorem]{Remark}
\newtheorem{fact}[theorem]{Fact}

\theoremstyle{definition}
\newtheorem{definition}[theorem]{Definition}

\def\sub{\subseteq}%

\newcommand{\iprod}[2]{\langle #1,#2 \rangle}

\DeclareMathOperator{\ill}{ill}

\DeclareMathOperator{\Ill}{ill}
\DeclareMathOperator{\Int}{int}
\DeclareMathOperator{\relint}{relint}

\DeclareMathOperator{\Sp}{span}
\DeclareMathOperator{\conv}{conv}
\DeclareMathOperator{\supp}{supp}

\DeclareMathOperator{\extreme}{ext}

\newcommand{\CC}{\mathbb{C}}
\newcommand{\RR}{\mathbb{R}}

\newcommand{\ZZ}{\mathbb{Z}}
\newcommand{\NN}{\mathbb{N}}

\newcommand{\TT}{\mathbb{T}}
\newcommand{\Sph}{\mathbb{S}}
\newcommand{\one}{\mathbbm{1}}

\def\S{\mathbb{S}}

\usepackage{xcolor}

\title{The complex Illumination problem}

\author{Liran Rotem, Alon Schejter and Boaz A. Slomka}
\date{}

\begin{document}
\maketitle
\begin{abstract}
We formulate a complex analog of the celebrated Levi-Hadwiger-Boltyanski  illumination (or covering) conjecture for complex convex bodies in $\CC^n$, as well as its (non-comparable) fractional version. 
	A key element in posing these problems is computing the classical and fractional illumination numbers of the complex analog of the hypercube, i.e., the  polydisc. 
	We prove that the illumination number of the polydisc in $\CC^n$ is equal to $2^{n+1}-1$ and that the fractional illumination number of the polydisc in $\CC^n$ is  equal to $2^n$.  
	In addition, we verify both conjectures for the classes of complex zonotopes and zonoids.
\end{abstract}

\section{Introduction}

\subsection{The classical illumination problem}
Let $K\sub\RR^n$ be a convex body, i.e., a  compact convex set with non-empty interior $\Int(K)$. Denote by $N(K)$ the minimal natural number $m$ such that $K$ can be covered by $m$ translates of its interior, that is 
$$
N(K)=\min\{m\in\NN\,:\, \exists x_1,\dots,x_m\in\RR^n,\,K\sub\bigcup_{i=1}^m(x_i+\Int(K))\}.
$$
A central and  long-standing problem in discrete geometry asks whether $N(K)\le 2^n$ for every convex body $K\sub\RR^n$ and whether $N(K)=2^n$ if and only if $K$ is an affine image of the hypercube $C_n$ in $\RR^n$. (It is easy to see that $N(C_n)=2^n$ as each two vertices of $C_n$ must be covered by different translates of its interior.) 

This problem, widely known as {\em the Levi-Hadwiger covering problem}, was posed in 1957 by Hadwiger \cite{hadwiger57}, a couple of years after its planar case had been   stated and settled by Levi \cite{Levi55}. In 1960, it was  independently posed by Gohberg and Markus \cite{GohMark60} in terms of covering a convex body by  smaller homothetic copies of it.

Yet another equivalent formulation of the Levi-Hadwiger covering problem is known as the {\em Boltyanski-Hadwiger illumination problem}, or simply the {\em illumination problem}.  Let $\partial K$ denote the boundary of a convex body $K\sub\RR^n$.  
A direction $v\in\RR^n$ is said to {\em illuminate} a point $x\in \partial K$  if the ray emanating from $x$ in direction $v$ passes through the interior of $K$, that is,  $x+tv\in\Int(K)$ for some $t>0$.  
We say that $K$ is illuminated by $V\sub\RR^{n}$ if each point in $K$ is illuminated by some $v\in V$. We  call the minimal number of directions needed to illuminate $K$ {\em the illumination number of $K$} and denote it  by $\ill(K)$. 
Clearly $v$ illuminates $x \in \partial K$ if and only if
$\frac{v}{|v|}$
illuminates $x$, so it is customary in the 
literature to assume that $v \in \S^{n-1}$, the unit sphere in $\RR^n$. For us however it will be convenient to allow illumination by general directions
in $\RR^n$. 

In \cite{boltyanski1960}, Boltyanski proved that $\ill(K)=N(K)$ for all convex bodies $K\sub\RR^n$ and restated the Levi-Hadwiger covering  problem as the following  illumination conjecture. 
\begin{conjecture*}
	Let $n\ge 2$. Given a convex body  $K\sub\RR^n$, we have $\ill(K)\le 2^n$. Moreover,  equality holds if and only if $K$ is an affine image of the hypercube. 
\end{conjecture*}
An equivalent formulation of the illumination problem was independently posed by  Hadwiger \cite{hadwiger1960},
in terms of the minimal number of ``light sources" in the exterior of a convex body $K\sub\RR^n$ that illuminate  $K$ entirely. In his definition, a light source, namely a point $y\in\RR^n$ in the exterior of $K$, is said to illuminate $x\in\partial K$ if  the ray emanating from $x$ in direction $x-y$ passes through $\Int(K)$ (so that illuminating directions in Boltyanski's definition can be thought of as light sources at infinity).

The  classical illumination (and covering) conjecture  has been studied extensively over the years and approached using tools and techniques from different mathematical areas. 

In particular, the conjecture has been verified  for various subfamilies of convex bodies, including smooth  bodies \cite{Levi55}, centrally-symmetric bodies in $\RR^3$ \cite{Lassak84},  bodies of constant width in high dimensions \cite{schramm_1988} and low dimensions \cite{bezdek_kiss_2009, BezLanNas07, bondarenko_prymak_radchenko_2022,lassak97,Weissbach96}, fat spindle bodies in high dimensions \cite{Bez12},
zonotopes and belt polytopes \cite{Martini87}, zonoids \cite{boltjanski92}, belt bodies \cite{Boltyanski95,BoltMart01,Bolt95}, $1$-symmetric bodies in $\RR^n$ for $n$  sufficiently large  \cite{tikhomirov_2017} and for all $n$ \cite{SunVrits24_1sym}, 
low-dimensional unconditional bodies \cite{SunVrit24_uncon}, centrally-symmetric cap bodies in $\RR^3$  \cite{IvaStra21} and higher dimensions 
\cite{BezIvaStra23},
convex hulls of  Minkowski sums of finite subsets of the lattice $\ZZ^n$ and $[-1/2,1/2]^n$ \cite{GaoMartWeZhang24}, and convex bodies close to the hypercube (with respect to the Hausdorff or the geometric distance)  \cite{LivTikh20}.

Other major advances made towards the resolution of the conjecture involve different bounds. 
In  \cite{Prymak2023}, it was shown that the illumination number of every convex body in $\RR^3$ is bounded from above by $14$ and that a bound of $96$ holds in $\RR^4$. Estimates for general bodies in $\RR^5$ and $\RR^6$ were given in \cite{diao}, and further low-dimensional results and improvements  are established in   \cite{ArmBond24}.

As for general bounds, the following  long-established estimates from \cite{ErdRog64} are essentially due to a result of Rogers \cite{Rogers57} and  the Rogers-Shephard inequality for the difference body \cite{RogShep57}:
$$
N(K)\le2^n(n\ln(n)+n\ln\ln(n)+5n)
$$
which holds for every centrally-symmetric  convex body $K\sub\RR^n$ and remains the best bound known for this class, and
$$
N(K)\le{\binom{2n}{n}}(n\ln(n)+n\ln\ln(n)+5n)=O(4^n\sqrt{n}\ln(n))
$$
which holds for every convex body $K\sub\RR^n$. Only recently, a sub-exponential improvement of this  general bound  have been established in \cite{Huang2021} using tools from asymptotic convex geometry, followed by almost exponential improvements given in \cite{campos2022hadwigers} and \cite{GalSing23}.

For additional historical remarks, related notions, problems and results, we refer the reader to the  surveys and monographs  \cite{Bez06, Bez10, BezKhan18, BolGoh85, BolGoh95, Boltyanski1997, BraMosPach06, MartSolt99}  and references therein.

\subsection{The fractional illumination problem}

The classical illumination problem admits a (weaker) fractional version, which was introduced and studied by Nasz\'{o}di \cite{Naszodi09}.  In this version, a non-negative Borel measure $\mu $ on $\RR^n$ is said to illuminate a convex body $K\sub\RR^n$ if for each $x\in K$, the set  $A_K(x)$ of all directions illuminating $x$ is of measure at least $1$, that is  $\mu(A_K(x))\ge 1$. The fractional illumination number of $K$ is then defined as 
$$
\ill^*(K)=\inf \{\mu(\RR^n)\,:\, \mu \text{ illuminates } K\}.
$$
Again, it is customary to assume that $\mu$ is supported on $\S^{n-1}$, but for ease of notation we prefer to allow
illumination by non-unit vectors. 

The fractional illumination conjecture  states that $\ill^*(K)<\ill^*(C_n)$ for any convex body $K\sub\RR^n$ whose every affine image is different than the hypercube $C_n$. 
Clearly, $\ill^*(K)\le\ill(K)$. Moreover, it is not hard to verify that $\ill^*(C_n)=\ill(C_n)=2^n$, and so the fractional illumination conjecture is  indeed  weaker than its classical counterpart.

Nasz\'{o}di  \cite{Naszodi09} proved  that $\ill^*(K)\le\binom{2n}{n}$ for any convex body $K\sub\RR^n$ and that $\ill^*(K)\le 2^n$ whenever $K$ is centrally-symmetric. As in the classical case, sub-exponential improvements of the general bound $\binom{2n}{n}$ follow from \cite{campos2022hadwigers,GalSing23,Huang2021}.

Covering numbers of convex bodies  also admit fractional extensions. Such  notions were introduced and studied in \cite{artsteinavidanraz2011} and \cite{artsteinavidan2013weighted}, where a weighted version  of the Levi-Hadwiger covering conjecture (in which only discrete covering measures are allowed)  was studied in the latter. 

Among other results, it was shown that the weighted  Levi-Hadwiger conjecture holds true for the class of centrally-symmetric convex bodies, which also settles the equality case in the fractional illumination conjecture for centrally-symmetric bodies. 
However, in general, the fractional  illumination problem remains open.

\subsection{The Complex illumination problems}

Motivated by the deep-rooted and rich theory of real convex bodies, various  results concerning complex convex bodies have been established in recent years, 
see e.g.,  \cite{Abardia2012, Abardia15, AbarBern11, AbarSaor15, AbarWann15, AbarBorDomoKert19,  Alesker01, Alesker03, BCK20, Bern09, BernFu11,  BernFuSol14, Cordero02, EllmHofst23, EllHofs24, glover2023stability, Harbel19, HH15,  K11, KKZ08, KoldPaouris13, KZ03, LiuWangHuang15, NT13, Oleszkiewicz2000,  Rotem14, Rubin10, Tckocz11, WangHe13, Zym08, Zym09}. 
In line with this program, the purpose of this article is to introduce and study complex analogs of the classical and fractional illumination problems.

The main objects considered to this end are complex convex bodies in $\CC^n$, namely unit balls of norms on $\CC^n$.  
These can  also be viewed as  centrally-symmetric convex bodies in $\RR^{2n}$ with additional symmetries due to the fact that a complex norm is invariant under complex rotations, i.e., multiplication by  $e^{\theta i}$ for  $\theta\in\RR$.  By identifying $\CC^n$ with $\RR^{2n}$ using the standard mapping $z=(x_1+iy_1,\dots, x_n+iy_n)\mapsto c(z)=(x_1,y_1,\dots,x_n,y_n)$, a  convex body $K\sub\RR^{2n}$ is said to be complex if $x\in K\implies e^{\theta i}x:=c(e^{\theta i}c^{-1}(x))\in K$ for every $\theta\in\RR$.

The definitions of  classical  and fractional illumination are carried over verbatim from the real case to the complex case. In fact, since (by our convention)  we identify between the unit sphere in $\CC^n$ and the unit sphere in $\RR^{2n}$, these definitions simply coincide.
However, as the hypercube is not a complex convex body,  we need to find a different candidate for the role of the body which is (conjectured to be the) hardest  to illuminate. Naturally, we consider the polydisc $D^n=D\times\dots\times D$ (where $D$ is the closed unit disk in $\CC$) for this role.  

Unlike the hypercube in $\RR^n$, whose illumination and fractional illumination numbers are both equal to $2^n$ and 
easy to compute, the situation with the polydisc is more complicated. Nevertheless, as we explain below in Section 
\ref{subsec:polydisc-intro}, we have $\ill^*(D^n)=2^n$ and $\ill(D^n)=2^{n+1}-1$. We therefore conjecture the following:

\begin{conjecture}\label{conj:complex_ill_conj}
Given  a complex convex body $K\sub\CC^n$, we have  $\ill(K)\le 2^{n+1}-1$. Moreover, equality holds if and only if  $K$ is a linear image of $D^n$.
\end{conjecture}

\begin{conjecture}\label{conj:frac_complex_ill_conj}
Given a complex convex body $K\sub\CC^n$, we have  $\ill^*(K)\le 2^{n}$. Moreover, equality holds if and only if  $K$ is a linear image of $D^n$.
\end{conjecture}

Note that Conjecture \ref{conj:frac_complex_ill_conj} is not weaker than or comparable to Conjecture \ref{conj:complex_ill_conj}, as opposed to  the situation in the real case, due to the fact that $\ill^*(D^n)<\ill(D^n)$.

\subsection{Main results}\label{sec:main_results}
\subsubsection{Illumination of the polydisc and related covering estimates} \label{subsec:polydisc-intro}

Our first results, proved in Section \ref{polydisc}, concern the illumination numbers of the polydisc.  The fractional illumination number of the polydisc is relatively easy to compute:
\begin{theorem}
 	\label{polydiscfracill}
 		For every $n\ge 1$ we have $\ill^*(D^n)=2^{n}$.
 \end{theorem}
For the classical illumination number more work is needed, but we prove:
\begin{theorem} 
	\label{polydiscclassicill}
	For every $n\ge 1$ we have 
	$\ill(D^n)=2^{n+1}-1$.
\end{theorem}
We mention immediately that Theorem \ref{polydiscclassicill} was already asserted in \cite{BoltyanskiMartini} by Boltyanski and Martini. Their main goal was to construct convex bodies $K$ and $L$ such that $\ill(K \times L) \ne \ill(K) \ill(L)$, as it is indeed clear from Theorem \ref{polydiscclassicill} that 
$\ill(D^{n+m}) < \ill(D^n) \ill(D^m)$. Unfortunately, it appears that the proof in \cite{BoltyanskiMartini} of
the upper bound $\ill(D^n) \le 2^{n+1}-1$ contains a delicate error. Our proof of this inequality is similar in spirit to the one from \cite{BoltyanskiMartini}, but a slightly different construction and a more careful analysis appear to be needed. See Remark \ref{rem:BMcomparison} for more details.

We prove Theorems \ref{polydiscfracill} and \ref{polydiscclassicill} by relating the illumination numbers 
of $D^n$ to the problem of covering $\TT^n$, 
the $n$-dimensional torus, by translates of the open hypercube $\left(0,\frac{1}{2}\right)^n$. 
Similar problems of covering $\TT^n$ by \emph{closed} hypercubes of the form \ $[0,\epsilon]^n$, were studied in 
\cite{MCELIECE1973119} and \cite{Bogdanov22}, and our proof borrows some ideas borrowed these works.

In fact, there is nothing special about cubes of side length $\frac{1}{2}$. In the fractional case we prove:
\begin{theorem}
	\label{torusfraccov}
	For any $\epsilon>0$, we have
	$$N^*\left(\TT^n,\left(0,\epsilon\right)^n\right)=\left( \frac{1}{\epsilon} \right)^n.$$
\end{theorem}
For classical covering we only have an exact result for special values of $\epsilon$: 
\begin{theorem}
	\label{torusclassiccov}
	For any $1\le m\in\NN$, we have
	$$N\left(\TT^n,\left(0,\frac{1}{m}\right)^n\right)=m^n+m^{n-1}+\cdots+m+1.$$
\end{theorem}
The precise definition of the covering numbers $N(\TT^n, (0,\epsilon)^n )$ and $N^*(\TT^n, (0,\epsilon)^n )$
is given in Section \ref{polydisc}. 

In the final section of this paper, Section \ref{othervalues}, we will explain how the more general covering numbers 
 $N(\TT^n,(0,\epsilon)^n)$ are also related to a certain illumination problem for the polydisk. We will also 
 discuss the computation of $N(\TT^n,(0,\epsilon)^n)$ in dimensions $n=2$ and $n=3$, using the results of 
 \cite{MCELIECE1973119} and \cite{Bogdanov22}.

\subsubsection{Illumination of Complex Zonotopes}

In Section \ref{zonotope} we show that Conjectures \ref{conj:complex_ill_conj} and \ref{conj:frac_complex_ill_conj} hold for the class of \emph{complex zonotopes}. Recall that a real zonotope is the Minkowski sum of segments. Up to translations we can take these segments to be centrally symmetric and consider only zonotopes of the form
$$\sum_{i=1}^N\left([-1,1]a_i\right)=
\left\{ \sum_{i=1}^N\sigma_i a_i \ :\  \sigma_i\in [-1,1] \text{ for all } 1 \le i\le N\right\} \subseteq \RR^n,$$
for  some $N\in \NN$ and vectors ${\{ a_i\}}_{i=1}^N$ in $\RR^n$. Hadwiger's conjecture was solved for zonotopes in \cite{Martini87}, with another proof sketched in \cite{boltjanski92}. As we explain see the inequality 
$\Ill(K) \le 2^n$ is easy for zonotopes, and most of the work concerns the characterization of the equality case.

As a natural analogy we define:
\begin{definition}
    A \emph{complex zonotope} is a Minkowski sum of complex discs, i.e. a set of the form:
    \[
    \sum_{i=1}^N\left(D a_i\right)=
\left\{ \sum_{i=1}^N\sigma_i a_i \ :\  \sigma_i\in D \text{ for all } 1 \le i\le N\right\} \subseteq \CC^n,
    \]
    for some $N\in \NN$ and vectors ${\{ a_i\}}_{i=1}^N \subseteq \CC^n$. 
\end{definition}
In Section \ref{zonotope} we prove:
\begin{theorem} \label{zonotopeclassic}
    Let $K \subseteq \CC^n$ be a full dimensional complex zonotope which is not a linear image of the polydisc $D^n$. 
    Then  $\ill(K)< \ill(D^n)$.
\end{theorem}
\begin{theorem} \label{zonotopefrac}
    Let $K \subseteq \CC^n$ be a full dimensional complex zonotope which is not a linear image of $D^n$. 
    Then  $\ill^\ast(K)< \ill^\ast(D^n)$.
\end{theorem}

Since $\ill^\ast(D^n) = 2^n < 2^{n+1}-1 = \ill(D^n)$, Theorem \ref{zonotopefrac} does not follow from Theorem
\ref{zonotopeclassic}. In fact our proofs will be very different: The proof of Theorem \ref{zonotopefrac}
will use the corresponding result for real zonotopes, while the the proof of Theorem \ref{zonotopeclassic} will 
rely on the strict inequality $\ill(D^{n+1})>2\cdot \ill (D^n)$ which has no equivalent in the real case.

\subsubsection{Illumination of Complex Zonoids}
In Section \ref{sec:zonoids} we establish that Conjectures 1.1 and 1.2 also hold for \emph{complex zonoids}, which
are Hausdorff limits of complex zonotopes. To be precise, recall that the Hausdorff distance between two 
(real or complex) convex bodies $K$ and $L$ is defined by
\[ d_H (K,L) = \min\left\{ \lambda >0 \ :\ K\subseteq L + \lambda B \text{ and } 
     L\subseteq K + \lambda B\right\}, \]
where $B$ is the unit Euclidean ball. We then say that $Z\subseteq \CC^n$ is a complex zonoid if there exists a 
sequence of complex zonotopes $\{ K_i \}_{i=1}^\infty \subseteq \CC^n$ such that $d_H(K_i, Z)\to 0$. 

Our result reads as follows:
\begin{theorem} \label{thm:zonoid-illumination}
    Let $Z \subseteq \CC^n$ be a full dimensional complex zonoid which is not a linear image of $D^n$. 
    Then  $\ill(Z)< \ill(D^n)$ and $\ill^\ast(Z)< \ill^\ast(D^n)$.
\end{theorem}

The real version of Theorem \ref{thm:zonoid-illumination} was established in \cite{boltjanski92}, 
using the corresponding result for zonotopes. Of course, in the real case the result for 
fractional illumination follows from the result for classical illumination, which is not true in the complex case. 
Nonetheless, the proof of \cite{boltjanski92} can be extended and adapted to prove Theorem \ref{thm:zonoid-illumination}.
We prefer however to present a new proof, which is  in many ways simpler. 
We present the proof in the complex case, though it works for the real case as well. 

\subsection*{Acknowledgements}  The first named author was funded by ISF grant no.~2574/24 and NSF-BSF grant no.~2022707. The second and third named 
authors were funded by ISF grant no.~784/20. We thank Yaron Ostrover and Pazit Haim-Kislev for useful conversations on complex convex bodies.
\section{Illumination of the polydisc}
\label{polydisc}

In this section, we compute the classical and fractional illumination numbers of the polydisc $D^n$, 
namely prove Theorems \ref{polydiscfracill} and \ref{polydiscclassicill}. 
To do so, we first show, in Section \ref{sec:equiv}, that these problems of illuminating the polydisc are equivalent to that of covering a torus by open cubes of side-length $\frac{1}{2}$. In  Section \ref{sec:torus}, we compute certain covering numbers of the torus by open cubes (of various side-lengths) via explicit constructions, which, in turn, implies our theorems.
\subsection{An equivalent formulation}\label{sec:equiv}
Our first task for the proof of Theorems \ref{polydiscfracill} and \ref{polydiscclassicill} is to relate the illumination of the polydisc to the covering of the torus by open cubes.

We first claim that in order to illuminate $D^n$, it is enough to illuminate the \emph{extremal points} of $D^n$.
Recall that a point $x\in\partial K$ is said to be an extremal point of $K$ if whenever $x=\lambda y+(1-\lambda)z$, for some $y,z\in K$ and $\lambda\in (0,1)$, it follows that $y=z=x$. 
Denote the set of extremal points of $K$ by ${\extreme}(K)$. We then have:
\begin{lemma}
\label{lem:extreme}
Let  $K\subseteq\RR^n$ be  a convex body. If a set of directions $V\subseteq \S^{n-1}$ illuminates ${\extreme}(K)$ then it illuminates $K$. 
\end{lemma}

It seems to us that Lemma  \ref{lem:extreme} should be well-known, however we could only find it stated in the literature for special classes of bodies (e.g., in  \cite{KissAndWet} for polytopes). 
For that reason and for the convenience of the reader, we provide a short proof for it. To this end, given a convex body $K\subseteq\RR^n$ and $x\in \partial K$ we denote by $N_x(K)=\{\theta\in \S^{n-1}\,:\, \iprod{\theta}{y-x}<0
\,\,\forall y\in K\}$ the
normal cone of $K$ at the point
$x$.
\begin{proof}[Proof of Lemma  \ref{lem:extreme}]
Fix  $x\in\partial K$, and let $H$ be a supporting hyperplane of $K$ passing through $x$. Since $\conv({\extreme}(K))=K$ (see e.g., \cite[Corollary 1.4.5]{Schneider2013}) the face $H\cap K$ contains a point $y\in {\extreme}(K)$.

Suppose that $v\in V$ illuminates $y$. This means that $\iprod{v}{\theta}<0$ for all $\theta\in N_y(K)$ (see e.g., \cite[Eq. (2.2)]{Schneider2013}. Since  $N_x(K)\subseteq N_y(K)$, it follows that $v$ illuminates $x$ as well.
\end{proof}

Going back to the polydisc, we see that it is enough to illuminate the $n$-dimensional torus:
\begin{equation*}\label{eq:ext_poly}
D_0^n:={\extreme}(D^n)=\partial D^1\times\dots\times \partial D^1.
\end{equation*}
It will be convenient for us to work with the flat torus $\TT^n = \RR^n / \ZZ^n$, which is a group with respect to the standard addition operation. For $a \in \TT^n$ and $\lambda \in \RR^n$, their addition $a + \lambda \in \TT^n$ is well-defined. For $a \in \TT^n$ and $k \in \NN$ we write $ka = a+a+\cdots+a\in \TT^n$. Note that $\lambda a$ is not well-defined if $a \in \TT^n$ and $\lambda \in \RR$ is not an integer.  

For $a,b\in \TT^1$, we define the oriented distance $d(a,b)$ as follows. Let $x,y\in \RR$ 
 be any representatives of $a$ and $b$ such that $x \le y < x+1$, and define $d(a,b) = y-x$. Note that $d$ is not symmetric, and in fact, for all $a,b\in\TT^1$  we have
 \begin{equation}\label{eq:sum_is_1}
     d(a,b) + d(b,a) =1.
 \end{equation} 
Given $x=(x_1, x_2,\ldots,x_n)\in \TT^n$, the open cube based at $x$ of side-length $\epsilon$ is 
\[ 
x + (0,\epsilon)^n = \{y\in\TT^n:\ 0 < d(x_i,y_i) < \epsilon \text{ for all } 1\le i\le n \}.
\]
The covering number of $\TT^n$ by $(0, \epsilon)^n$ is defined as the minimal number of translates of $(0, \epsilon)^n$ whose union contains $\TT^n$, that is 
\[
N(\TT^n, (0, \epsilon)^n):=\min\left\{N\in\NN\,:\,\exists x^1,\dots x^N\in\TT^n
,\;\TT^n = \bigcup_{i=1}^N\left(x^i+(0,\epsilon)^n\right)\right\}.
\]
Similarly the fractional covering number of $\TT^n$ by $(0, \epsilon)^n$ is defined as
\[
N^*(\TT^n,(0, \epsilon)^n):=\inf\left\{\mu(\TT^n)\,:\,\forall\,x\in\TT^n,\,\mu(x-(0, \epsilon)^n)\ge1 \right\}
\]
where the infimum is taken over all non-negative Borel measures $\mu$ on $\TT^n$. 
\begin{proposition}\label{prop:torus_polydisc_ill}
Let $n\ge2$. We have $\ill(D^n) = N(\TT^n, (0,\frac{1}{2})^n)$ and   $\ill^*(D^n) = N^*(\TT^n, (0,\frac{1}{2})^n)$.
\end{proposition}
\begin{proof}[Proof of Proposition   \ref{prop:torus_polydisc_ill}]
Our goal is to establish that the illumination problem can be solved by illuminating points on $D_0^n$ by illuminating directions which are also on $D_0^n$. We already know from \ref{lem:extreme} that it is sufficient to illuminate $D_0^n$, so we are to show the same for the illuminating directions.
 
Observe that due to the product structure of the polydisc,  a non-empty subset $A\subset D_0^n$ is illuminated by a vector $v=(v_1,\dots, v_n)\in \CC^{n}$  if and only if  $v_i \ne 0$ for all $i$ and $A$ is illuminated by the extremal vector $f(v):=(v_1/|v_1|,\dots,v_n/|v_n|)\in D_0^n$. Indeed, by definition, $x=(x_1,\dots, x_n)\in A$ is illuminated by $v$ if $x+tv\in\Int(D^n)=\Int(D^1)\times\dots\times\Int(D^1)$ for some small enough $t>0$. Clearly, this occurs if and only if $x_i+tv_i\in\Int(D^1)$ for each $i\in\{1,\dots,n\}$ or, in other words, each component $x_i$ on the boundary of the corresponding $i^{\rm th}$ disc $D^1$ is illuminated by the vector $v_i$, independently of other components (which, in particular, means that $v_i\neq 0$). 

The above observation implies that if a measure $\mu$ illuminates $D^n$, we may assume $\mu$ supported on $\{v\in \CC^n\,:\,v_i\ne 0\,\, \forall i\in\{1,\ldots,n\}\}$. Moreover, the  pushforward of $\mu$ by $f$, namely the measure $\mu\circ f^{-1}$ on $D_0^n$, is an illumination measure of $D^n$ whose total mass is the same as $\mu$'s. 

Having it clear know that it is sufficient to illuminate $D_0^n$ by vectors of $D_0^n$, we continue denoting by $Y_v$ the set of points in $D_0^n$ which are illuminated by $v\in D_0^n$. We conclude that 
\begin{equation}\label{eq:ill_by_torus}
\ill(D^n)=\min\left\{N\in\NN\,:\,\exists\, v^1\dots,v^N\in D_0^n,\,\,\,\bigcup_{i=1}^N Y_{v^i}=D_0^n\right\}
\end{equation}
and
\begin{equation}\label{eq:ill*_by_torus}
\ill^*(D^n)=\inf\{\mu(D_0^n)\,:\,\mu\ge0,\,\,\mu(\{v \in D_0^n\,:\,\ x\in Y_v\})\ge1\,\,\,\forall x\in D_0^n\}. 
\end{equation}
Note that $Y_v$ is  given by 
\begin{equation}\label{eq:open_cube}
Y_v=\left\{x=(x_1,\dots,x_n)\in D_0^n\,:\,\, \iprod{x_i}{v_i}<0
\,\,\,\,\forall\,\, i\in\{1,\dots,n\}\right\}.
\end{equation}
Indeed,  as described above, illuminating boundary points of $D_0^n$ by $v\in D_0^n$ is equivalent to simultaneously illuminating  boundary points on the unit disc $D^1$ by $v_1,\dots,v_n\in S^1$. Since each direction $v_i\in S^1$ clearly illuminates the open half circle centered at $-v_i$, we obtain \eqref{eq:open_cube}.

Finally, consider the  bijection $f:D_0^n\to\TT^n$ mapping each  $v=(e^{2\pi\theta_1 i},\dots,e^{2\pi\theta_n i})$, where $\theta_1\dots,\theta_n\in [0,1)$, to $(\theta_1,\dots,\theta_n)$. It is easy to verify that $f(Y_{-v})$ is an open cube of side-length $\frac12$  which is  based at $(\theta_1,\dots,\theta_n)-(\frac14,\dots,\frac14)$. Therefore, it follows by \eqref{eq:open_cube} that a  point $v\in D_0^n$ illuminates $x\in D_0^n$ if and only if the open cube $f(Y_{-v})$ of side-length $\frac12$ contains $f(x)$. Consequently, \eqref{eq:ill_by_torus} implies that $\ill(D^n)=N(\TT^n,(0,\frac12)^n)$ and \eqref{eq:ill*_by_torus} implies that $\ill^*(D^n)=N^*(\TT^n,(0,\frac12)^n)$, as claimed. 
\end{proof}

\subsection{Covering the torus by open cubes}\label{sec:torus}

In light of Proposition   \ref{prop:torus_polydisc_ill}  we are interested in computing  $N(\TT^n,(0,\frac{1}{2})^n)$ and 
$N^*(\TT^n,(0,\frac{1}{2})^n)$. We will actually compute the more general $N(\TT^n,(0,\frac{1}{m})^n)$ for $m \in \NN$ and 
$N^*(\TT^n,(0,\epsilon)^n)$ for all $0 < \epsilon < 1$. In Section \ref{othervalues} we will discuss the more general problem of finding $N(\TT^n,(0,\epsilon)^n)$ and explain its relation to illumination.

The problem of covering $\TT^n$ with \emph{closed} cubes was studied in \cite{MCELIECE1973119} and in \cite{Bogdanov22} where various bounds on these covering numbers were established.  Exact values were computed in  \cite{MCELIECE1973119} for dimension $2$, and in \cite{Bogdanov22} for dimension $3$. In particular, their results imply that  $N(\TT^2,(0,\frac{1}{2})^2)=7$ and  $N(\TT^3,(0,\frac{1}{2})^3)=15$.

We begin by establishing a lower  bound for the classical covering number $N(\TT^n, (0,\frac{1}{m})^n)$. For $m=2$
this was already proved in \cite{BoltyanskiMartini} in the equivalent language of illumination numbers, and 
their proof is similar to our own:

\begin{proposition}
\label{prop:lowerboundm}
For every $m \in \NN,m \ge 2$ we have $$N(\TT^n,(0,\frac{1}{m})^n)\ge \frac{m^{n+1}-1}{m-1}=m^n+m^{n-1}+\cdots+1.$$  
\end{proposition}
\begin{proof}
Define $a_n = N(\TT^n,(0,\frac{1}{m})^n)$. We define $\TT^0$ to be a single point, and therefore set $a_0 = 1$. By induction it is enough to show that $a_n \ge m a_{n-1}+1$ for all $n\ge1$.

Let $V$ be a set of $a_n$ translations of the cube $(0,\frac{1}{m})^n$ which cover $\TT^n$. Let $C_0 = x + (0,\frac{1}{m})^n$ be an arbitrary element of $V$. Define $U = \{x_1 + \frac{k}{m}\ :\ 1 \le k \le m\} \subseteq \TT^1$. Finally, for every $u\in U$ define $\TT_u = \{u\}\times \TT^{n-1}$ and $V_u = \{ C \in V\ :\ C \cap \TT_u \ne \emptyset \}$.

For every $u \in U$ the set $\{ C\cap \TT_u\ :\ C \in V_u \}$ is a set of $(n-1)$-dimensional cubes of side length $\frac{1}{m}$ which covers $\TT_u$. Since $\TT_u$ is an $(n-1)$-dimensional torus, we have by definition $|V_u| \ge a_{n-1}$.

For every $u, u' \in U$ we have $d(u_1, u'_1) \ge \frac{1}{m}$. Therefore no cube in $V$ can intersect both $\TT_u$ and $\TT_{u'}$, i.e. $V_u \cap V_{u'} = \emptyset$. Moreover, by construction $C_0 \notin V_u$ for any $u\in U$. Therefore 
\[
a_n = |V| \ge \left|\{C_0\} \cup \bigcup_{u \in U} V_u \right| = 1 + \sum_{u\in U} |V_u| \ge 1 + ma_{n-1},
\]
as claimed. 
\end{proof}

\begin{remark} Proposition   \ref{prop:lowerboundm} can be extended to cubes of arbitrary side length $\epsilon > 0$. The proof in the case $\epsilon = \frac{1}{m}$ is simpler, so we postpone the discussion of the general case to Section \ref{othervalues}.
\end{remark}

For the matching upper bound we first need two simple facts regarding the distance on $\TT^1$:
\begin{fact}\label{fact:cyclic}
Every finite set $A \subseteq \TT^1$ may be put in \emph{cyclic order}, i.e., one can write $A=\{a_1, a_2, \ldots, a_k\}$ so that for every $1\le i < j \le k$,
\[ d(a_i, a_j) = \sum_{\ell = i}^{j-1} d(a_\ell, a_{\ell+1}).\]
\end{fact}
\begin{proof}
    For each $a \in \TT^1$ let $x(a)$ be its representative in $[0,1)$.  
    Arrange the elements of $A$ such that $x(a_1) < x(a_2) < \cdots x(a_n)$. Since $d(a_i, a_j) = x(a_j) - x(a_i)$ for all $i<j$, the claim follows. 
\end{proof}
Note that by \eqref{eq:sum_is_1}, if we define $a_{k+1} = a_1$ then $\sum_{\ell=1}^k d(a_\ell, a_{\ell +1}) = 1$. Also note that the cyclic order is indeed cyclic, i.e., if $a_1, a_2, \ldots, a_k$ is a cyclic order of $A$ so is $a_2,a_3,\ldots,a_k,a_1$.

\begin{fact}\label{fact:scaling}
If $a,b \in \TT^1$ and $d(a,b) < \frac{1}{k}$ for some  $k \in \NN$, then $d(ka, kb) = k \cdot d(a,b)$.
\end{fact}
\begin{proof}
Let $x$ and $y$ be representatives of $a$ and $b$ such that $d(a,b)=y-x < \frac{1}{k}$. Therefore $kx\le ky < kx+1$ so,  by definition,  $d(ka,kb)=ky-kx=k\cdot d(a,b)$.    
\end{proof}

We can now prove:
\begin{proposition}
\label{prop:upperboundm}
For every $\epsilon > 0$ and $m \ge 2$ we have $N(\TT^n,(0,\frac{1}{m})^n) \le \frac{m^{n+1}-1}{m-1}$.   
\end{proposition}

\begin{remark} \label{rem:BMcomparison}
As stated in the introduction, Proposition \ref{prop:upperboundm} was previous stated by Boltyanski and Martini in 
\cite{BoltyanskiMartini}, for $m=2$ and in the equivalent language of illumination numbers. Unfortunately, it appears
that their proof contains a delicate error. In our language, the authors state in Example 3 that 
$N(\TT^2, (0,\frac{1}{2})^2) \le 7$ since  $\TT^2 = \bigcup_{k=0}^6 (x^k + (0,\frac{1}{2})^2)$, 
where $x^k = \left( \frac{k}{3}, \frac{k}{7} \right) \in \TT^2$. This is not true, as e.g. the point 
$(\frac{5}{6}, \frac{3}{14}) \in \TT^2$ is not covered by any of these cubes.

The issue in the proof is that the authors claim that the set 
$\{ x^k_1, x^{k+1}_1, x^{k+2}_1\} \subseteq \TT$ contains $3$ equally spaced points for every $0\le k \le 6$, where the vectors are 
taken cyclically. However
$\{ x^6_1, x^{7}_1, x^{8}_1\} = \{ x^6_1, x^0_1, x^1_1\} = \{0,0,\frac{1}{3}\}$, so the proof breaks down. In Theorem 4
of \cite{BoltyanskiMartini} the authors attempt to prove the result for general $n$ by induction on $n$, but the 
same issue affects not only the base case $n=2$ but also the induction step.

In our proof we also choose an explicit set of cubes, just like in \cite{BoltyanskiMartini}, but our set is different.
Our choice does not seem to allow a simple inductive proof, and instead we require a very careful analysis.
\end{remark}

\begin{proof}[Proof of Proposition \ref{prop:upperboundm}.] For $0\le k\le n$, set $b_{k}=\frac{m^{n+1-k}-1}{m-1}$. Define  
$U_{0}=\left\{ \frac{k}{b_{0}}\,:\,\ 1\le k\le b_{0}\right\}$,
and consider the set of points 
\[
V=\left\{ (u,mu,m^{2}u,\ldots.m^{n-1}u)\,:\, u\in U_{0}\right\} \subseteq\TT^{n}.
\]
As $\left|V\right|=\frac{m^{n+1}-1}{m-1}$, it is enough to show that the set of cubes $\left\{ v+(0,\frac{1}{m})^{n}\,:\,\ v\in V\right\} $ covers $\TT^{n}$. 

Towards this goal, fix $x\in\TT^{n}$. We will find a cube that covers $x$, working coordinate by coordinate: First we find a set 
$U_1\subseteq U_0$ of size $|U_1| \ge b_1$ such that $x_1\in u+\left(0,\frac{1}{m}\right)$ for every $u\in U_1$. Then, we find a set 
$U_2\subseteq U_1$ of size $|U_2| \ge b_2$ such that $x_2\in mu+\left(0,\frac{1}{m}\right)$ for every $u\in U_2$, and so on. After the last
iteration we will have at least one candidate $u\in U_0$ such that $x_k\in m^ku+\left(0,\frac{1}{m}\right)$ for all $1\le k\le n$, which 
means $x\in v+{\left(0,\frac{1}{m}\right)}^{n}$ for the corresponding $v\in V$.

To formally carry out this scheme, we inductively define a sequence of sets $U_{0}\supseteq U_{1}\supseteq\cdots\supseteq U_{n}$ such that for all $1\le k\le n$ we have:
\begin{enumerate}[labelindent=\parindent,leftmargin={*},label=(P\arabic*),align=left]
\item\label{enu:set-size}$\left|U_{k}\right|=b_{k}$.
\item\label{enu:set-bound}For all $u\in U_{k}$ we have $d\left(m^{k-1}u,x_{k}\right)<\frac{1}{m}$.
\item\label{enu:set-separation}For all $u\ne u'$ in $U_{k}$ we have
$d(m^{k}u,m^{k}u')\ge\delta_{k}$, where $\delta_{k}=m^{k}\frac{m-1}{m^{n+1}-1}$. 
\end{enumerate}
Note that properties \ref{enu:set-size} and \ref{enu:set-separation} hold for $k=0$: to see property \ref{enu:set-separation} simply observe that the distance between different elements of $U_{0}$ is at least $\frac{1}{b_{0}}=m\delta_{0}>\delta_{0}$. To describe the inductive step, assume we have already constructed $U_{k-1}$. Consider
the set $m^{k-1}U_{k-1}\cup\left\{ x_{k}\right\} \subseteq\TT^{1}$ and arrange it in cyclic order according to Fact \ref{fact:cyclic}, say 
\[
m^{k-1}u_{1},m^{k-1}u_{2},m^{k-1}u_{3},\ldots,m^{k-1}u_{b_{k-1}},x_{k}
\]
 (recall that every cyclic permutation of the cyclic order is still cyclic, so we may assume for ease of notation that $x_{k}$ is ``last'' in the list). We now simply choose $U_{k}=\left\{ u_{r},u_{r+1},u_{r+2}\ldots,u_{b_{k-1}}\right\} $, where $r=b_{k-1}-b_{k}+1$. 

Clearly $\left|U_{k}\right|=b_{k-1}-r+1=b_{k}$. For the second property, we fix $r\le i\le b_{k-1}$ and use the cyclic order to compute
\begin{align*}
d\left(m^{k-1}u_{i},x_{k}\right) & =1-d\left(x_{k},m^{k-1}u_{i}\right)=1-\left(d(x_{k},m^{k-1}u_{1})+\sum_{j=1}^{i-1}d(m^{k-1}u_{j},m^{k-1}u_{j+1})\right)\\
 & \le1-\sum_{j=1}^{i-1}d(m^{k-1}u_{j},m^{k-1}u_{j+1})\le1-i\delta_{k-1}\le1-r\delta_{k-1}.
\end{align*}
 A direct computation shows that  
\[
(1-r\delta_{k-1})-\frac{1}{m}=-\frac{(m^{k}+1)(m-1)}{m\left(m^{n+1}-1\right)}<0,
\]
so we indeed have $d\left(m^{k-1}u,x_{k}\right)\le1-r\delta_{k+1}<\frac{1}{m}$ for all $u\in U_{k}$, which is \ref{enu:set-bound}. 

For the third and final property of $U_{k}$, note that since for all $r\le i,j\le a_{k-1}$ we have 
\[
d\left(m^{k-1}u_{i},m^{k-1}u_{j}\right)\le d\left(m^{k-1}u_{r},x_{k}\right)<\frac{1}{m},
\]
 it follows from Fact \ref{fact:scaling} that $d\left(m^{k}u_{i},m^{k}u_{j}\right)=m\cdot d\left(m^{k-1}u_{i},m^{k-1}u_{j}\right)$.
In particular, the elements
\[
m^{k}u_{r},m^{k}u_{r+1},m^{k}u_{r+2}\ldots,m^{k}u_{b_{k-1}}
\]
 are still given in cyclic order. Clearly in order to lower bound the distance between all pairs it is enough to bound the distance between (cyclically) consecutive pairs.  For all $r\le i\le b_{k-1}-1$
we have 
\[
d\left(m^{k}u_{i},m^{k}u_{i+1}\right)=m\cdot d\left(m^{k-1}u_{i},m^{k-1}u_{i+1}\right)\ge m\delta_{k-1}=\delta_{k},
\]
so all that remains is to bound $d\left(m^{k}u_{b_{k-1},}m^{k}u_{r}\right)$.
Towards this goal we again compute using the cyclic order:
\begin{align*}
d\left(m^{k}u_{r},m^{k}u_{b_{k-1}}\right) & =m\cdot d\left(m^{k-1}u_{r},m^{k-1}u_{b_{k-1}}\right)=m\left(1-d\left(m^{k-1}u_{b_{k-1}},m^{k-1}u_{r}\right)\right)\\
 & =m\left(1-d\left(m^{k-1}u_{b_{k-1}},m^{k-1}u_{1}\right)-\sum_{i=1}^{r-1}d\left(m^{k-1}u_{i},m^{k-1}u_{i+1}\right)\right)\\
 & \le m\left(1-\delta_{k-1}-\sum_{i=1}^{r-1}\delta_{k-1}\right)=m(1-r\delta_{k-1}).
\end{align*}
Hence 
\begin{align*}
d\left(m^{k}u_{b_{k-1},}m^{k}u_{r}\right) & =1-d\left(m^{k}u_{r},m^{k}u_{b_{k-1}}\right)\ge1-m\left(1-r\delta_{k-1}\right)\\
 & =(m^{k}+1)\cdot\frac{m-1}{m^{n+1}-1}>\delta_{k},
\end{align*}
which proves property \ref{enu:set-separation} and completes our inductive construction of the sets $\left\{ U_{k}\right\} _{k=0}^{n}$. 

Since $\left|U_{n}\right|=b_{n}=1$, we can write $U_{n}=\left\{ u\right\} $.
Then for all $1\le k\le n$ we have $u\in U_{k}$, and so $d\left(m^{k-1}u,x_{k}\right)<\frac{1}{m}$.
This exactly means that $x\in(u,mu,m^{2}u,\ldots.m^{n-1}u)+(0,\frac{1}{m})^{n}$, as claimed.
\end{proof}

We now turn our attention to the fractional case, which is much simpler.
\begin{proposition}\label{prop:fractionalcover}
For every $0<\epsilon<1$ we have $N^*(\TT^n, (0,\epsilon)^n) = \left(\frac{1}{\epsilon}\right)^n$.
\end{proposition}
\begin{proof}
For the upper bound, it is enough to note that the uniform measure on $\TT^n$ with total measure $\left(\frac{1}{\epsilon}\right)^n$ is a covering measure. 

For the lower bound we argue similarly to Proposition   2.10 of \cite{artsteinavidan2013weighted}. Let $\mu$ be a covering measure of $\TT^n$ by $(0,\epsilon)^n$. This means that $\mu * \one_{(0,\epsilon)^n} \ge \one$, where the convolution $*$ is with respect to the group structure on $\TT^n$. Integrating over the torus we see that $\mu(\TT^n) \epsilon^n \ge 1$. 
\end{proof}

The proof of Theorems \ref{polydiscfracill} and \ref{polydiscclassicill} is now immediate: 
\begin{proof}[Proof of Theorems \ref{polydiscfracill} and \ref{polydiscclassicill}]
The fact that $\Ill(D^n) = 2^{n+1}-1$ follows from Propositions \ref{prop:torus_polydisc_ill}, \ref{prop:lowerboundm}, and \ref{prop:upperboundm}. The fact that $\Ill^*(D^n) = 2^{n}$ follows from Propositions \ref{prop:torus_polydisc_ill} and \ref{prop:fractionalcover}. 
\end{proof}

\section{Illumination of Zonotopes}
\label{zonotope}
In this section we discuss the illumination problem for complex zonotopes and  prove Theorems \ref{zonotopeclassic} 
and  \ref{zonotopefrac}. For the real case, an equivalent theorem was proven in \cite{Martini87}, with another proof 
sketched in \cite{boltjanski92}. We begin by explaining the proof of \cite{boltjanski92}, presenting 
it in a slightly different way that will be convenient for the proof of the complex case. We then prove 
Theorems \ref{zonotopeclassic} and  \ref{zonotopefrac}. As we will see, the theorems are independent of each other and 
require essentially different ideas. 
\subsection{Real Zonotopes}
\label{sec:real_zonotopes}
Recall that a (real) zonotope $K$ is the Minkowski sum of finitely many segments in $\RR^n$. Up to a translation of $K$ we may 
assume that these segments are symmetric around the origin, say
$[-a_1, a_1]$, $[-a_2, a_2]$, $\ldots$ ,$[-a_N, a_N]$ for $a_1,a_2,\ldots,a_N \in \RR^n$. 
We then say that $K$ is \emph{induced} by $a_1, a_2,\ldots,a_N$ and write
\begin{equation}\label{eq:real_zon_form}
K = Z_R(a_1, a_2, \ldots, a_N) = \left\{ \sum_{j=1}^N \lambda_j a_j  :\ |\lambda_j| \le 1 \text{ for }1\le j \le N \right\}.
\end{equation}
We will always assume that $K$ is full-dimensional so, in particular,  $N \ge n$. Using the fact that $\Int(K) = \bigcup_{0 < t < 1} (tK)$ it is immediate that 
\begin{equation} \label{eq:int-zonotope}
\Int (K) = \left\{ \sum_{j=1}^N x_j a_j  :\ |x_j| < 1 \text{ for }1\le j \le N \right\}.
\end{equation}
It is also clear that 
\begin{equation} \label{eq:ext-zonotope}
\extreme(K) \subseteq \left\{ \sum_{j=1}^N x_j a_j  :\ |x_j| = 1 \text{ for }1\le j \le N \right\},
\end{equation}
and we denote the right hand side of \eqref{eq:ext-zonotope} by $Z'_R(a_1,a_2,\ldots,a_n)$. 

In \cite{Martini87} Martini resolved the Hadwiger conjecture for zonotopes:
\begin{theorem}[Martini] \label{thm:real_zonotope_ill}
Let $K \subseteq \RR^n$ be a zonotope that is not a linear image of the cube. Then $\ill^*K\le\ill K\le 3\cdot 2^{n-2}<2^n$. 
\end{theorem}

We will now reproduce the proof of Theorem \ref{thm:real_zonotope_ill}. The proof is essentially the one in \cite{boltjanski92}, with small simplifications,
but our presentation and the lemmas we prove will be useful for the complex case as well. 

We begin by proving that the illumination number of $Z_R(a_1,a_2,\ldots,a_N)$ depends only on the directions of the vectors 
$a_1, a_2,\ldots,a_N$. We will need the following fact, see e.g.,  \cite[Lemma 1]{boltjanski92}:

\begin{fact}
\label{sum_convex}
Assume $K = K_1 + K_2$ for $K_1, K_2 \subseteq \RR^n$ with $K_1$ being full-dimensional. The every measure $\mu$
which illuminates $K_1$ also illuminates $K$. In particular $\ill(K) \le \ill(K_1)$ and
$\ill^\ast (K) \le \ill^\ast (K_1)$.
\end{fact}

We now prove:

\begin{lemma}
\label{lem:normilizing_zonotope}
Let $a_1, a_2 ,\ldots, a_N \in \RR^n$  span
the whole space. Then a measure $\mu$  on $\RR^n$ illuminates $Z_R(a_1, \ldots, a_N)$ if and only if it illuminates $Z_R(r_1 a_1, \ldots, r_N a_N)$ for any $r_1, r_2, \ldots, r_N > 0$.
In particular, the bodies $Z_R(a_1, \ldots, a_N)$ and $Z_R(r_1 a_1, \ldots, r_N a_N)$ have the same illumination number and the same
fractional illumination number. 
\end{lemma}
\begin{proof}
Define $r:=\max_{1\le j\le N}r_j$ . Since $a_j=\frac{r_j}{r}a_j+\frac{r-r_j}{r}a_j$ for all $j$ we have
$$Z_R(a_1,\cdots,a_N)=Z_R\left(\frac{r_1}{r}a_1,\cdots,\frac{r_N}{r}a_N\right)+Z_R\left(\frac{r-r_1}{r}a_1,\cdots,\frac{r-r_N}{r}a_N\right).$$
If $\mu$ illuminates $Z_R(r_1 a_1, \ldots, r_N a_N)$ it also illuminates the dilation
$Z_R(\frac{r_1}{r} a_1, \ldots, \frac{r_N}{r} a_N)$, and then by Fact \ref{sum_convex} it also illuminates $Z_R(a_1,\cdots,a_N)$. The reverse implication is the same.

By considering only measures on the form $\mu = \sum_{j=1}^k \delta_{v_j}$ it follows that  $Z_R(a_1, \ldots, a_N)$ and 
$Z_R(r_1 a_1, \ldots, r_N a_N)$ have the same illumination number. By considering all Borel measures we see that they also have the same
fractional illumination number. 
\end{proof}

Next, we show we can restrict our attention to zonotopes of a special form:
\begin{lemma}\label{lem:non_cube_zonotope_form}
Let $K\subseteq \RR^n$ be a zonotope which is not a linear image of the cube. Then there exists a zonotope $K'=Z_R(a_1, a_2, \ldots, a_{n+1})$
with the following properties:
\begin{enumerate}
\item The vectors  $a_1, a_2, \ldots, a_n$ form a basis of $\RR^n$.
\item For any $1 \le j \le n$, $a_{n+1}$ is not parallel to $a_j$.
\item There exists a linear dependence of the form 
\begin{equation} \label{eq:relation} \lambda_1 a_1 + \lambda_2 a_2 + \cdots + \lambda_{n+1} a_{n+1} = 0 
\end{equation}
with $\lambda_j \in \{0,1\}$ for $1 \le j \le n-2$ and $\lambda_{n-1} = \lambda_n = \lambda_{n+1} = 1$.
\item We have $\ill(K) \le \ill(K')$ and $\ill^*(K) \le \ill^*(K')$. 
\end{enumerate}
\end{lemma}

\begin{proof}
Write $K  = Z_R(b_1, b_2, \ldots, b_N)$. We can always assume no two $b_j$'s are parallel, since otherwise we can replace them by 
a single vector. Since $K$ is full-dimensional the set $\{b_1, b_2, \ldots, b_N\}$ contains a basis of $\RR^n$ and by  permuting 
the indices we may assume that $b_1, b_2, \ldots, b_n$ is such a basis. Since $K$ is not a linear image of the cube we in 
fact have $N \ge n+1$.

Clearly there exists a linear dependence 
\[ \eta_1 b_1 + \eta_2 b_2 + \cdots \eta_{n+1} b_{n+1} = 0. \]
Since no two vectors are parallel we know that $\eta_j \ne 0$ for at least three values of $j$. One of these values must be $j = n+1$. Since $b_1, \ldots, b_n$ is a basis, and by permuting indices again we may assume $ j = n-1$ and $j=n$ are also such indices. By replacing $b_j$ by $-b_j$ if necessary, which doesn't change $K$, we may assume that $\eta_j \ge 0$ for all $j$.  

We now define $a_j = \eta_j b_j$ if $\eta_j \ne 0$, and $a_j = b_j$ otherwise and set $K' = Z_R(a_1, a_2, \ldots, a_{n+1})$. Clearly the first three properties are satisfied. For the forth note that by Lemma \ref{lem:normilizing_zonotope} and Fact \ref{sum_convex} we have 
\[ 
\ill(K') = \ill\left( Z_R(b_1, \ldots, b_{n+1})\right) \ge
\ill(\left(Z_R(b_1, \ldots, b_{n+1}) + Z_R(b_{n+2}, \ldots, b_N)\right) = \ill(K), 
\] 
and similarly $\ill^*(K') \ge \ill^*(K)$. 
\end{proof}

To explain the importance of  Lemma \ref{lem:non_cube_zonotope_form},  assume we are given a zonotope $K = Z_R(a_1, \ldots, a_{n+1})$ of the form given by the lemma.
Given a vector $x = \sum_j x_j a_j \in \partial K$, and a direction $v = \sum_j v_j a_j$, 

we need to know when $x + tv \in \Int(K)$.
Of course, since $x + tv = \sum_j (x_j + t v_j)a_j$, a sufficient condition is that $|x_j + t v_j| < 1$ for all $j$, but this is not
a necessary condition. Instead we may fix $\delta \in \RR$ and use \eqref{eq:relation} to write
\[ x + t v = \sum_{j=1}^{n+1} (x_j + tv_j + \delta \lambda_j) a_j, \] 
a procedure that we call a \emph{refinement by $\delta$}. Therefore in order to prove that $x + tv \in \Int(K)$ it is enough to find any 
$\delta \in \RR$ such that $|x_j + tv_j + \delta \lambda_j | < 1$ for all $1\le j \le n+1$.

We are now ready to prove Theorem \ref{thm:real_zonotope_ill}. We again stress that this is essentially the proof of \cite{boltjanski92}, 
presented slightly differently:

\begin{proof}[Proof of Theorem \ref{thm:real_zonotope_ill}]
The inequality $\ill^*K\le \ill K$ is always true. To prove that $\ill K\le 3\cdot 2^{n-2}$ we first apply Lemma \ref{lem:non_cube_zonotope_form} and assume
without loss of generality that $K = Z_R(a_1, a_2, \ldots, a_{n+1})$ is of the form asserted by the lemma. 

We will find an explicit set of vectors $Y$ with $|Y| \le 3\cdot 2^{n-2}$ which illuminates $K$. By \eqref{eq:ext-zonotope} and Lemma \ref{lem:extreme} it 
is enough to illuminate $K_0 = Z'_R(a_1, a_2, \ldots, a_{n+1})$. Define
$$F=\{\sigma_1a_1+\cdots+\sigma_{n-2}a_{n-2}\ :\ |\sigma_j|=1\text{ for all }j\}$$
and 
$$V=\left\{\frac13(a_{n-1}-a_n),\frac13(a_n-a_{n+1}),\frac13(a_{n+1}-a_{n-1})\right\}.$$
We claim that $Y=F+V$, which clearly satisfies $|Y| \le 3\cdot 2^{n-2}$, illuminates $K_0$. To see this fix $x=x_1a_1+\cdots x_{n+1}a_{n+1}\in K_0$, i.e. $|x_j|=1$ for all $j$. There are two cases:

\paragraph{Case I:} There exists $\alpha$ with $|\alpha|=1$ such that $x_{n-1}=x_n=x_{n+1}=\alpha$.  

This is the simpler case -- we simply choose $f=-x_1a_1 - \cdots - x_{n-2}a_{n-2}\in F$, and (arbitrarily) $v=\frac13(a_{n-1}-a_n)\in V$. Then, performing refinement by $\delta=-\frac {\alpha}2$ we see that 
\begin{align*}
x + (f+v) &= \sum_{j=1}^{n+1} \left(x_j + f_j + v_j - \frac{\alpha}{2} \lambda_j\right)a_j  \\
        &= \sum_{j=1}^{n-2} \left(-\frac{\alpha}{2}\lambda_j a_j\right) + \left(\frac{1}{3} + \frac{\alpha}{2}\right)a_{n-1} +
            \left(-\frac{1}{3} + \frac{\alpha}{2}\right)a_{n} + \frac{\alpha}{2} a_{n+1}
\end{align*}
(passing from the first to the second line used the fact that $\lambda_{n-1}=\lambda_n=\lambda_{n+1} = 1$).
Since all coefficients have absolute value $<1$ we see that $x+(f+v) \in \Int(K)$, so $f+v$ illuminates $x$.

\paragraph{Case II:} There exists $\alpha$ with $|\alpha|=1$ such that one of $x_{n-1},x_n,x_{n+1}$ is equal to $\alpha$ and the other two are equal to $-\alpha$.

We again choose $f=-x_1a_1 - \cdots - x_{n-2}a_{n-2}\in F$, but this time the choice of $v$ is not arbitrary: If $i\in\{n-1,n,n+1\}$ 
is the index with $x_i=\alpha$, we choose $v$ to be the unique element of $V$ of the form 
$v=\frac{1}{3}\left(\alpha a_j-\alpha a_{i}\right)$ with $j\ne i$. For ease of notation let us assume that  $\alpha=1$ and $i=n$, so
$x_n=1$, $x_{n-1}=x_{n+1}=-1$ and $v=\frac13(a_{n-1}-a_n)$ -- the other cases work in the same way. 

Performing refinement by $\delta=\frac{\alpha}6$ we see that
\begin{equation*}
x + (f+v) = \sum_{j=1}^{n+1} \left(x_j + f_j + v_j + \frac{\alpha}{6} \lambda_j\right)a_j  =
    \sum_{j=1}^{n-2} \left(\frac{1}{6}\lambda_j a_j\right) - \frac{1}{2}a_{n-1} +
    \frac{5}{6} a_{n} - \frac{5}{6} a_{n+1}.
\end{equation*}
Again all coefficients have absolute value $ < 1$, so $x+(f+v)\in \Int(K)$ and $f+v$ illuminates $x$.
All other choices for $\alpha$ and $i$ work in exactly the same way, with the same refinement by $\frac{\alpha}{6}$.
\end{proof}

\begin{remark}\label{rem:sharp_bd_Re_zon}
One can check that if $H \subseteq \RR^2$ is a hexagon and $K = H \times [-1,1]^{n-2} \subseteq \RR^n$ then $\Ill(K) = 3\cdot 2^{n-2}$. 
Therefore the bound given by  Theorem \ref{thm:real_zonotope_ill} is sharp. 
\end{remark}

\subsection{Classical illumination of complex zonotopes}
\label{zonotopeclassicsection}

We begin with some notation and simple observations concerning complex zonotopes and their illumination, which are analogous to those made in Section
\ref{sec:real_zonotopes} for real zonotopes. 

Recall that a complex zonotope $K\sub\CC^n$ is, up to a translation, the Minkowski sum of finitely many complex discs $A_1,\dots,A_N$ in $\CC^n$. Each of these discs $A_j$ is represented by some vector $a_j\in\CC^n$ such that  $A_j=\{re^{\theta i}a_j\,:\, 0\le r\le1,\,\,0\le\theta<2\pi\}$. We  say that $K$ is induced by $a_1,\dots,a_N$ and write 
\begin{equation*}
	K=Z_C(a_1,\dots,a_N)=\left\{\sum_{j=1}^N x_ja_j\,:\, x_j\in\CC,\,|x_j|\le1\text{ for } 1\le j\le N\right\}.
\end{equation*}
Clearly, $Z_C(a_1,\dots,a_N)=Z_C(e^{\theta_1 i}a_1,\dots,e^{\theta_N i}a_N)$ for any choice of $\theta_1,\dots,\theta_N\in\RR$.

As in the case of real zonotopes, we will always assume that $K$ is full-dimensional and hence, in this case, we have 
\begin{equation}\label{eq:comp_zon_int}
\Int(K)=\left\{\sum_{j=1}^N x_ja_j\,:\, x_j\in\CC,\,|x_j|<1\text{ for } 1\le j\le N\right\}
\end{equation}
and 
\begin{equation}\label{eq:comp_zon_ext}
\extreme(K)\sub\left\{\sum_{j=1}^N x_ja_j\,:\, x_j\in\CC,\,|x_j|=1\text{ for } 1\le j\le N\right\}.
\end{equation}
As in the real case, we denote the right hand side of \eqref{eq:comp_zon_ext} by $Z'_C(a_1,\dots,a_N)$.

The following lemmas are complex analogues of Lemma \ref{lem:normilizing_zonotope} and Lemma \ref{lem:non_cube_zonotope_form}, respectively. We omit their proofs  as they are identical to the proofs of their real counterparts, with the exception of using a rotation of $b_j$ by $\frac{\overline{\eta}_j}{|\eta_j|}$ instead of its reflection through the origin, in order to assume without loss of generality  that $\lambda_1,\dots,\lambda_{n+1}$ are all real non-negative numbers.

\begin{lemma}\label{lem:normalized_complex_zonotope}
Let $a_1, a_2 ,\ldots, a_N \in \CC^n$  span
the whole space. Then a measure $\mu$  on $\CC^n$ illuminates $Z_C(a_1, \ldots, a_N)$ if and only if it illuminates $Z_C(r_1 a_1, \ldots, r_N a_N)$ for any $r_1, r_2, \ldots, r_N > 0$.
In particular, the bodies $Z_C(a_1, \ldots, a_N)$ and $Z_C(r_1 a_1, \ldots, r_N a_N)$ have the same illumination number and the same
fractional illumination number. 
\end{lemma}

\begin{lemma}\label{lem:non_poly_zonotope_form}
Let $K\sub\CC^n$ be a complex zonotope with non-empty interior which is not a linear image of a polydisc. Then there exists a complex zonotope $K'=Z_C(a_1,\dots,a_{n+1})$ with the following properties:
\begin{enumerate}
	\item The vectors $a_1,\dots,a_n$ form a basis of $\CC^n$.
	\item For any $1\le j\le n$,  $a_{n+1}$ is not parallel to $a_j$.
	\item There exists a linear dependence of the form 
	\begin{equation}\label{eq:comp_zon_dep}
		\lambda_1a_1+\dots+\lambda_{n+1}a_{n+1}=0
	\end{equation}
with $\lambda_j\in\{0,1\}$ for all $\le j\le n-2$, and $\lambda_{n-1}=\lambda_n=\lambda_{n+1}=1$.
\item We have $\ill(K)\le\ill(K')$ and $\ill^*(K)\le\ill^*(K')$.
\end{enumerate}
\end{lemma}

Despite the similarities demonstrated in this section between real and complex zonotopes, it is not clear how 
to adapt the proof of Theorem \ref{thm:real_zonotope_ill} to the complex case. For example, that proof used 
strongly the fact that the extremal set of a real zonotope is finite, which is not true for complex zonotopes. We thus provide an inherently different proof for Theorem \ref{zonotopeclassic}, as follows:

\begin{proof}[Proof of Theorem \ref{zonotopeclassic}]
Let $n\ge 2$ and fix  a complex zonotope with non-empty interior   $K\sub\CC^n$ and which is not a linear image of a polydisc. Applying Lemma \ref{lem:non_poly_zonotope_form}, we assume without loss of generality that $K = Z_C(a_1, a_2, \ldots, a_{n+1})$ is of the form asserted by the  lemma.
By Theorem \ref{polydiscclassicill}, we have
\begin{equation}\label{ill_com_zon_twice_dim}
\ill(D^n)=2^{n}-1> \ill(D^n)-1=2\ill(D^{n-1})
\end{equation}
and hence it suffices to prove that $\ill(K)\le 2\ill(D^{n-1})$. 
To achieve this goal, we shall construct two  sets $V_1,V_3\sub\CC^n$, each of which of cardinality $\ill(D^{n-1})$, such that $V_1\cup V_3$ illuminates all the boundary points of $K$ which lie in $Z'_C(a_1,\dots,a_{n+1})$.  Note that by \eqref{eq:comp_zon_ext} and Lemma \ref{lem:extreme},  $V_1\cup V_3$ would  also illuminate $K$. To construct our illuminating sets, 
let $H_1:=\Sp\{a_1,\dots,a_{n-1}\}$, $H_2:=\Sp\{a_1,\dots,a_{n-2},a_{n}\}$ and define: 
$$K_1:=K(a_1,\dots,a_{n-1})\sub H_1,$$
$$K_2:=K(a_1,\dots,a_{n-2},a_{n})\sub H_2.$$
Since $\Sp\{a_1,\dots,a_{n}\}=\CC^n$, $K_1$ and $K_2$ are each a linear image of an $(n-1)$-dimensional polydisc. For $i\in\{1,2\}$, let $V_i\sub H_i$ be a minimal illuminating set of $K_i$ relative to $H_i$. This means that  $V_i$ has cardinality $\ill (D^{n-1})$, and that  for every $x:=x_1a_1+\dots+x_{n-1}a_{n-1}\in K_i$ there exists  $v=v_1a_1+\dots + v_{n-1}a_{n-1}\in V_i$ which illuminates $x$ as a boundary point of $K_i$. Equivalently, the latter means that there exists $t>0$ such that $|x_j+tv_j|<1$ for all  $j\neq n+1-i$.

While the set $V_1$  will be used ``as is" to illuminate our complex zonotope $K$, the set $V_2$ will require an alteration by shifting each one of its element. Namely, for each  $v=v_1a_1+\dots+v_{n-2}a_{n-2}+v_na_n\in V_2 $ define
\begin{equation}
\label{v_2}
    v':=v_1a_1+\dots+v_{n-2}a_{n-2}+v_na_n-v_{n}a_{n+1}
\end{equation}
and denote $V_3=\{v'\,:\,v\in V_2\}$. Clearly $|V_3|=|V_2|=\ill(D^{n-1})$.

We claim that  $V_1\cup V_3$ illuminates $K$. To see this,  fix a point $x=x_1a_1+\cdots+x_{n+1}a_{n+1}$ in $Z'_C(a_1,\dots,a_{n+1})$ and consider the following two cases. \\

\noindent\textbf{Case I:}  $x_{n}\ne -x_{n+1}$. 

 Since $x-(x_na_n-x_{n+1}a_{n+1})=x_1a_1+\dots+x_{n-1}a_{n-1}\in K_1$, 
it follows by the definition of $V_1$ that there exists $v\in V_1$ and $t>0$ such that 
\begin{equation}\label{eq:comp_zon_case1}
|x_j+tv_j|<1 \text{ for all }1\le j<n.
\end{equation}
Next we show that  $x+tv\in\Int(K)$. By \eqref{eq:comp_zon_int}, this means that we need to write $x+tv$ in the form $y_1a_1+\dots+y_{n+1}a_{n+1}$ with $|y_j|<1$ for all $1\le j\le n+1$. Note that we cannot simply choose $y_j=x_j+tv_j$ for all $j$ as we have no guarantee that either  $|x_{n}+tv_n|<1$ or $|x_{n+1}+tv_{n+1}|<1$. However, we may preform a refinement using the linear dependence $\lambda_1a_1+\dots+\lambda_{n+1}a_{n+1}=0$ given in \eqref{eq:comp_zon_dep} to write  $x+tv$ in the desired form. Indeed, since  $x_n\neq -x_{n+1}$ we can find some
$z\in\CC$ with $|z|=1$ such that both $\langle x_{n},z \rangle<0$ and $\langle x_{n+1},z \rangle<0$. This means, on the one hand, that  we can choose $\delta>0$ small enough so that both $|x_n+\delta z|<1$ and $|x_{n+1}+\delta z|<1$. On the other hand, by \eqref{eq:comp_zon_case1}, we can  choose such  $\delta$  so that also 
$|x_j+tv_j+\delta z\lambda_j|<1$ for all $1\le j<n$.
Recall that  $\lambda_n=\lambda_{n+1}=1$. Therefore, as illustrated in  Figure \ref{fig:zontopefirstcase}, we have 
\begin{align*}
	x+tv
	&=   x+tv+\delta z(\lambda_1a_1+\dots+\lambda_{n+1}a_{n+1})\\
	&= (x_1+tv_1+\delta z\lambda_1)a_1+\cdots + (x_{n-1}+tv_{n-1}+\delta z\lambda_{n-1})a_{n-1}\\
	&+(x_{n}+\delta z)a_{n} +(x_{n+1}+\delta z)a_{n+1}=y_1a_1+\dots+y_{n+1}a_{n+1}\in \Int (K).
\end{align*}

\begin{figure}[ht!]
	\centering
	\includegraphics[width=0.9\linewidth]{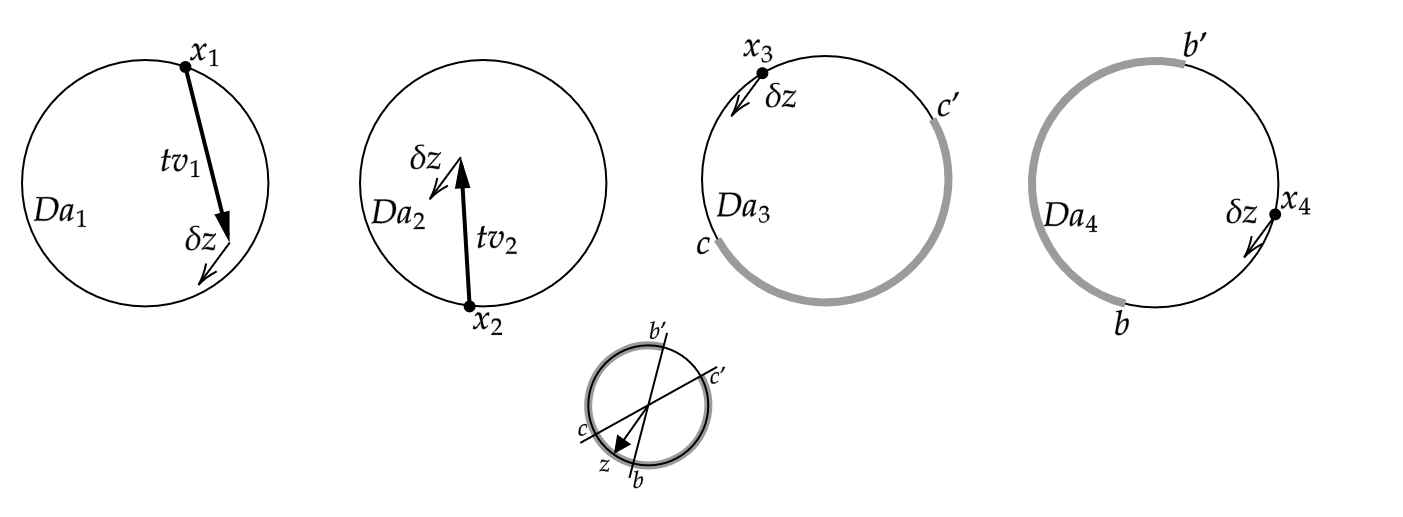}
	\caption{A visualization, for $n=3$, of the fact that if $x_3\neq -x_4$ then any $v\in V_1$ which illuminates $x_1a_1+x_2a_2\in K_1$ also 
	illuminates $x=x_1a_1+x_2a_2+x_3a_3+x_4a_4\in K$. Here $Da_j=K(a_j)$ is the  disc induced by $a_j$ and $z$  is a direction  illuminating both $x_3\in Da_3$ and $x_4\in Da_4$. Using  a small refinement by $\delta z$ and \eqref{eq:comp_zon_dep} we write  $x+tv=x+tv+\delta z(\lambda_1a_1+\dots+\lambda_{4}a_{4})$ in the form of \eqref{eq:comp_zon_int}. 
	}	\label{fig:zontopefirstcase}
\end{figure}

\noindent\textbf{Case II:} $x_{n}= -x_{n+1}$.

In this case,  we will show that $V_3$ illuminates $x$.  Our first step towards this goal is to  pick some $v\in V_2$ which illuminates the point
$
x_1a_1+\dots+x_{n-2}a_{n-2}+x_na_n\in K_2
$
and $t>0$ for which $|x_j+tv_j|<1$ for all  $j\in\{1,2,\dots,n-2,n\}$. By the definition of $v'$ in \eqref{v_2} and the fact that $x_{n+1}=-x_n$, we have 
\begin{align}\label{eq:comp_zon_case2}
	|x_j+tv'_j|<1 \text{ for all }  j\in\{1,\dots,n+1\},\, j\neq n-1
\end{align}
Our second and last step is to show that $x+tv'\in\Int(K)$ by putting it in the form of \eqref{eq:comp_zon_int}  using the linear dependence \eqref{eq:comp_zon_dep}. To do so, let $z\in\CC$ be some unit vector for which $\langle x_{n-1},z\rangle<0$, and $\delta>0$ small enough so that $|x_{n-1}+\delta z|<1$.  By \eqref{eq:comp_zon_case2}, we may choose such  $\delta$ so that also  
\begin{align*}
	|x_j+tv'_j+\delta z\lambda_j|<1 \text{ for all }  j\in\{1,\dots,n+1\},\, j\neq n-1.
\end{align*}
Recalling that  $\lambda_{n-1}=1$, we therefore have (as illustrated in Figure \ref{fig:zonotopesecondcase})
\begin{align*}
x+tv' &=   x+tv'+\delta z(\lambda_1a_1+\dots+\lambda_{n+1}a_{n+1})\\
&= \sum_{\substack{j=1 \\ j\neq n-1}}^{n+1}(x_j+tv'_j+\delta z\lambda_j)a_j+(x_{n-1}+\delta z)a_{n-1} \in \Int (K),
\end{align*}
as needed. 
\begin{figure}[ht!]
    \centering
    \includegraphics[width=0.9\linewidth]{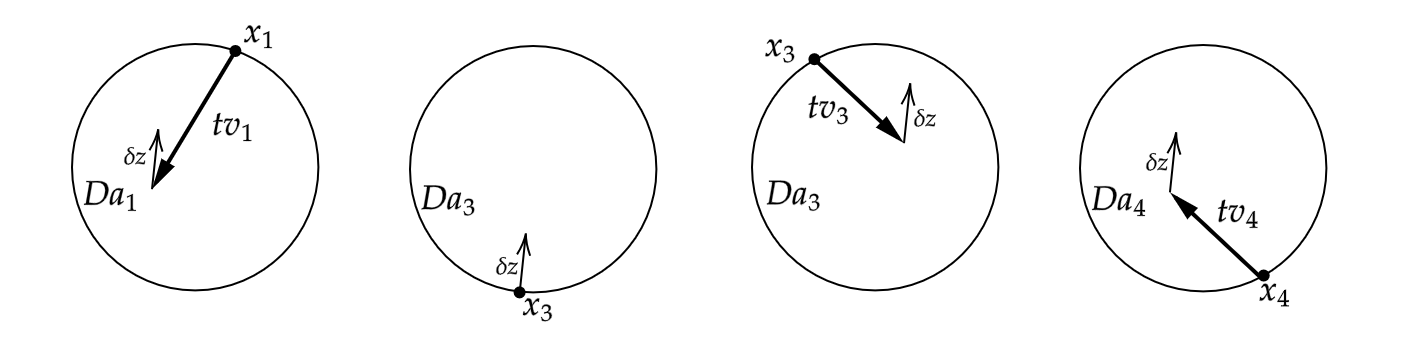}
    \caption{A visualization, for $n=3$, of the fact that if $x_3=-x_4$ and $v\in V_2$ illuminates $x_1a_1+x_3a_3\in K_2$, then $v'\in V_3$ illuminates $x=x_1a_1+x_2a_2+x_3a_3+x_4a_4\in K$. Here $Da_j=K(a_j)$ is the disc induced by $a_j$ and $z$ is a direction which illuminates $x_3\in Da_3$. Using a small refinement by $\delta z$ and \eqref{eq:comp_zon_dep} we can write $x+tv'=x+tv'+\delta z (\lambda_1a_1+\dots+\lambda_4 a_4)$ in the form of \eqref{eq:comp_zon_int}.
}\label{fig:zonotopesecondcase}
\end{figure}
\end{proof}

\begin{remark}
As opposed to the  case of real zonotopes (see Remark \ref{rem:sharp_bd_Re_zon} above), we do not know whether the bound provided in the proof of Theorem \ref{zonotopeclassic} is sharp, that is,  if there exists a complex zonotope $K\sub\CC^n$ such that $\ill(K)=\ill(D^n)-1$.
\end{remark}

\subsection{Fractional illumination of complex zonotopes}
\label{zonotopefracsection}

In this section we prove Theorem \ref{zonotopefrac},  thus confirming the fractional version of the complex illumination conjecture for the class of complex zonotopes.

Denote by $\Sp_\RR\{a_1,\dots,a_N\}$  the span of complex vectors  $a_1,\dots,a_N\in\CC^n$ over the field $\RR$, that is the set of all linear combinations of $a_1,\dots,a_N$ with real coefficients. 
As in \eqref{eq:real_zon_form} we also denote the real zonotope in  $\Sp_\RR\{a_1,\dots,a_N\}$ which is  induced by the complex vectors $a_1,\dots,a_N$, and thus embedded into $\CC^n$,
by
$$ 
Z_R(a_1, a_2, \ldots, a_N) = \left\{ \sum_{j=1}^N \lambda_j a_j \,:\,\lambda_j\in\RR\text{ and } |\lambda_j| \le 1 \text{ for }1\le j \le N \right\}.
$$
Clearly,  $Z_R(a_1,\dots,a_N)$ is a lower-dimensional subset of    $Z_C(a_1,\dots,a_N)$.
More generally, any translate $Z$  of $Z_R(a_1,\dots,a_{N})$ (and therefore  any Minkowski sum of $N$ line segments in $\CC^n$) is a real zonotope.
We say that  $y\in \CC^n$ illuminates $x\in Z$ if $x+ty\in\relint(Z)$ for some $t>0$. 

As mentioned in the introduction, our proof  of Theorem \ref{zonotopefrac} relies on the solution of the classical illumination problem for real zonotopes, described in Section \ref{sec:real_zonotopes}.

The main idea of the proof  is  to construct a continuous illumination measure for a given complex zonotope $Z=Z_C(a_1,\dots,a_N)$ out of  a classical finite illuminating set $Y$ of the real zonotope $Z_R(a_1,\dots,a_N)$. This is done by employing the properties of complex zonotopes to show that for every $x\in Z'_C(a_1,\dots,a_N)$ there exists a subset  of full measure $\Theta\sub(0,\pi)$ such that for every $\theta\in \Theta$, the set $e^{\theta i}Y$ illuminates $x$.  Note that this does not mean that we can find $\theta$ such that  $e^{\theta i}Y$ illuminates  $Z$, as $\Theta$ depends on the point $x$. However,  as captured in the following lemma, to obtain a fractional illumination of $Z$  one can circumvent this obstacle by uniformly spreading each point mass in $Y$ onto a half circle.

To describe this construction  precisely, let $Y\sub\CC^n\setminus\{0\}$ be finite. Denote by $U$ the uniform probability measure on $(0,\pi)$. Given  $y\in Y$ for some finite set $0\not\in Y$, consider the function  $\pi_y:(0,\pi)\to\CC^n$  defined by $\pi_y(\theta)=e^{\theta i}y$,  whose image we denote by $V_y$. Let $\mu_y=U\circ \pi^{-1}_y$ be the push-forward of $U$ onto $V_y$ and define the measure  $\mu_Y:=\sum_{y\in Y} \mu_y$. 

\begin{lemma}\label{lem:frac_ill_theta}
	Let $K\sub\CC^n$ be a complex convex body.  Suppose that for each $x\in\extreme(K)$ there exists $\Theta_x\sub(0,\pi)$ with  
$U(\Theta_x)=1$ such that  
 $e^{\theta i}Y$  illuminates $x$ for
 every $\theta\in\Theta_x$.
Then $\mu_Y$ illuminates $K$. 
\end{lemma}

\begin{proof}
Fix $x\in\extreme(K)$. 
For each $y\in Y$ denote by $\Theta_{x,y}\sub\Theta_x$ the subset of all $\theta\in \Theta_x$ such that $e^{\theta i}y$ illuminates $x$. We thus have,
$$\mu_Y(Y_x)=\sum_{y\in Y}\mu_y(Y_x)=\sum_{y\in Y}\mu_y(\pi_y(\Theta_{x,y}))=\sum_{y\in Y}U(\Theta_{x,y})\ge
U\Big(\bigcup_{y\in Y}\Theta_{x,y}\Big)\ge
U(\Theta_x).
$$
By Lemma \ref{lem:extreme}, it follows that $\mu_Y$ is  an illumination measure of $K$.
\end{proof}

We are now ready to prove Theorem \ref{zonotopefrac}.

\begin{proof}[Proof of Theorem \ref{zonotopefrac}] 
We first apply Lemma \ref{lem:non_poly_zonotope_form} to assume, without loss of generality, that $K=Z_C(a_1,\dots,a_{n+1})$ is of the form asserted by the lemma.

Consider $H=\Sp_\RR\{a_1,\dots,a_n\}\sim \RR^n$. Since the coefficients in the linear dependence  \eqref{eq:comp_zon_dep} are all real numbers, it follows that $K_R:=Z_R(a_1,\dots,a_{n+1})$ is a real zonotope with non-empty interior in $H$, which is not a linear image of a cube. Therefore, by Theorem \ref{thm:real_zonotope_ill}, there exists a set $Y\sub H$ of at most $3\cdot 2^{n-2}$ directions  which illuminates $K_R$.

Our goal is to show  that $\mu_Y$ illuminates $K$. Since the total mass of $\mu$ is $\mu(\CC^n)=|Y|<2^n$,  this will imply that $\ill^*(K)<2^n$ and thus complete our proof. 

 Let $x=x_1a_1+\dots+x_{n+1}a_{n+1}\in Z'_C(a_1,\dots,a_{n+1})$ be a boundary point of $K$,  and fix any
$\theta\in(0,\pi)$ such that  $\langle e^{\theta i}, x_j \rangle\ne 0$ for all $1\le j\le n+1$ (note that almost every $0<\theta<\pi $ satisfies this condition). 
By \eqref{eq:comp_zon_ext} and Lemma \ref{lem:frac_ill_theta}, it is enough to show that $e^{\theta i}Y$ illuminates $x$ in order to conclude that $\mu_Y$ illuminates $K$. In other words, it is enough to show that there exist $y\in Y$  and $t>0$ such that 
\begin{equation}\label{eq:theta_y}
x+te^{\theta i}y	\in \Int(K).
\end{equation}
To this end, denote
$$
A_j:= D_j \cap \ell_j^{\theta},\,\,j\in\{1,\dots,n+1\}
$$
where $D_j=Z_C(a_j)$ is the disc induced by $a_j$  and $\ell_j^{\theta}=\{x_ja_j+t e^{\theta i}a_j \,:\,t\in \RR\}$ is the line passing through   $x_ja_j$ in direction $e^{\theta i }a_j$. Since $\langle e^{\theta i}, x_j \rangle\ne 0$, it follows that $A_j:= D_j \cap \ell_j^{\theta}$ is a proper segment, see Figure \ref{fig:realcutofcomplex}.  
Consider the real zonotope $A=A_1+\dots+A_{n+1}$. It  is not hard to verify (again, see Figure \ref{fig:realcutofcomplex}) that
$\relint(A_j)\sub \relint(D_j)$,  and that $A$  is a translation of the centrally-symmetric real zonotope $K'_R=Z_R(\frac{|A_1|}{2|a_1|}e^{\theta i}a_1,\dots,\frac{|A_{n+1}|}{2|a_{n+1}|}e^{\theta i}a_{n+1})$. 
\begin{figure}[ht!]
	\centering
	\includegraphics[width=1\linewidth]{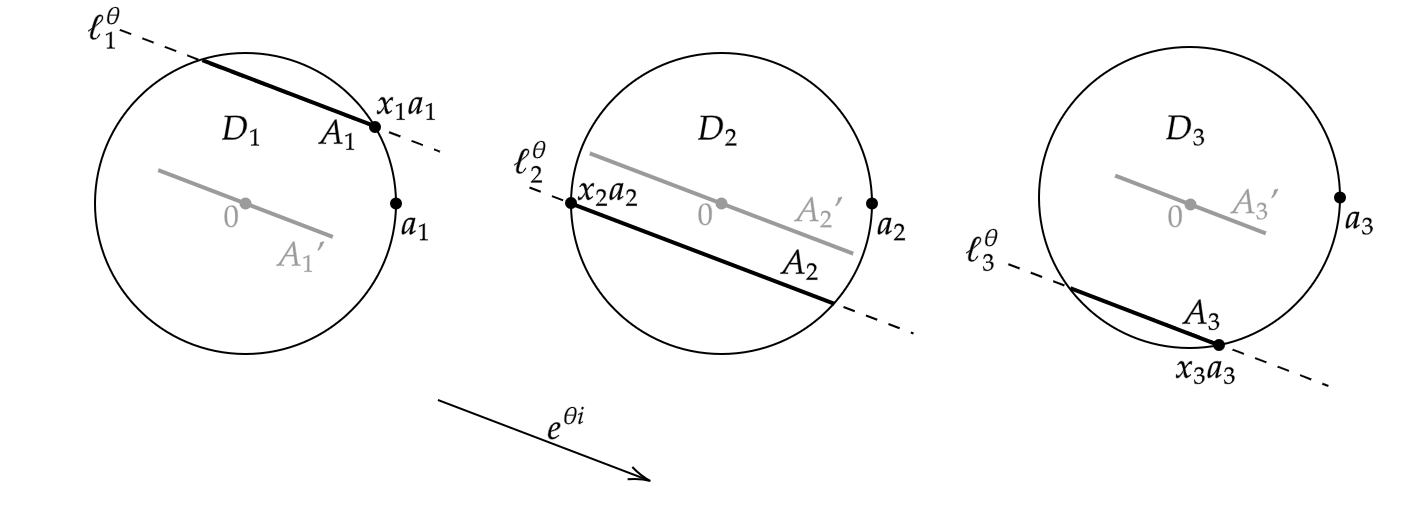}
	\caption{A representation of a complex zonotope $K=D_1+D_2+D_3\sub\CC^2$ by three circles.  
Given a point $x=x_1a_1+x_2a_2+x_3a_3\in K_0$ and a direction $e^{\theta i}$ which is not orthogonal to any vector in  $\{x_1,x_2,x_3\}$, the resulting segments $A_i=\ell^\theta_i\cap D_i$ are proper.  Also notice that their sum  $A$ can be  translated to the centrally-symmetric real zonotope $A'$ (the sum of the segments $A'_1,A'_2,A'_3$ in grey), both of which are contained in $K$.  		
}
	\label{fig:realcutofcomplex}
\end{figure}
Therefore, it follows by \eqref{eq:int-zonotope} and \eqref{eq:comp_zon_int} that
$$
\relint(A)=\relint(A_1)+\dots+\relint(A_{n+1})\sub\relint(D_1)+\dots+\relint(D_{n+1})=\Int(K).
$$
Moreover, Lemma  \ref{lem:normilizing_zonotope} tells us that  $K'_R$ has the same illuminating sets as 
$$Z_R(e^{\theta i}a_1,\dots,e^{\theta i}a_{n+1})=e^{\theta i}Z_R(a_1,\dots,a_{n+1}),$$
and so $e^{\theta i}Y$ illuminates both  $K'_R$ and $A$. Since $x\in A$, we can thus   $y\in Y$ and $t>0$ such that 
$x+te^{\theta i}y\in\relint(A)\sub\Int(K),
$ which establishes \eqref{eq:theta_y} and completes  our proof.  
\end{proof}

\section{Complex Zonoids and their illumination}
\label{sec:zonoids}

\subsection{Complex Zonoids}
\label{sec:zonoid_definition}
As in the real case, we define a complex zonoid to be the Hausdorff limit of complex zonotopes. 
Recall that we always translated our zonotopes to be origin-symmetric, so our zonoids will be origin-symmetric as well, although this is not essential. 

Just like in the real case, there are many equivalent definitions for complex zonoids: 
\begin{theorem} \label{the:mainzonoid}
The following are equivalent for a complex convex body $Z \subseteq \CC^n $:
\begin{enumerate}
\item $Z$ is the Hausdorff limit of complex zonotopes.

\item \label{itm:measure-zonoid} There exists a finite Borel measure $\mu$ on $S =\{x\in\CC^n\ |\ \|x\|= 1\}$ such that the support function of $Z$ is:
$$h_Z(\theta)= \int_{S}|{\langle x,\theta \rangle}_\CC | d\mu(x)$$

\item There exists an isometry:
$$T:B(Z^\circ)\to L^1(X,\nu)$$
for some measure space $(X, \nu)$, where $B(Z^\circ)$ is the  Banach space for which the polar $Z^\circ$ is the unit ball, and $L^1$ is over the complex field.

\item $Z$ is the $w^*$-continuous linear image of the closed unit ball $B(L^\infty[0,1])$.

\item \label{itm:abs-cont-zonoid} There exists an absolutely continuous vector-valued function  $\psi:[0,\ell]\to\CC^n$, parameterized by arc-length, 
such that 
\[ 
Z = \left\{ \int_0^\ell g(s)\psi'(s)ds\ :\ g:[0,\ell]\to\CC \text{ is measurable with } \|g\|_{L^\infty} \le 1\right\}. 
\] 
\end{enumerate}
\end{theorem}

In the real case it is well-known that these definitions are equivalent: the equivalence of the the first four definitions 
is given in \cite{Bolker}, and the final definition appears in \cite{Boltyanski1997}. We will
not give a full proof of Theorem \ref{the:mainzonoid} in the complex case as it is fairly similar to the real case and is beyond the scope of this paper.
Instead we will only prove that the first definition implies the second, as this is the only part of the theorem we  
will need in order to study the illumination number of complex zonoids.

First, let us explain our notation. Since we have a natural identification $\CC^n \cong \RR^{2n}$, we actually have 
two inner products on $\CC^n$: The standard complex inner product on $\CC^n$, and the standard real inner product on 
$\RR^{2n}$. To avoid confusion we denote them by $\langle x, y \rangle_\CC$ and $\langle x, y \rangle_\RR$ respectively.
The two inner products are related as follows.
\begin{lemma} \label{real_complex_ip}
For all $x,y \in \CC^n$ we have
$$\left |{\langle x,y \rangle}_\CC\right|= \frac{1}{4} \int_0^{2\pi}\left|{\langle e^{it} x,y \rangle}_\RR\right|dt.$$
\end{lemma}
\begin{proof}
Note that for any unitary map $U: \CC^n \to \CC^n$, replacing $x$ and $y$ by $Ux$ and $Uy$ keeps the identity unchanged. 
We may therefore assume without loss of generality that $y = (r,0,0\ldots,0)$ and 
$x = (r_1 e^{\theta_1 i}, \ldots, r_n e^{\theta_n i})$. We then have 
\begin{align*}
\int_0^{2\pi}\left|{\left\langle e^{ti} x,y \right\rangle}_\RR\right| dt &= 
\int_0^{2\pi}\left|{\left\langle (r_1 e^{\left(\theta_1+t\right) i}, \ldots, r_n e^{\left(\theta_n+t\right) i}), 
    (r,0,0\ldots,0) \right\rangle}_\RR\right| dt \\
& = \int_0^{2\pi} | r r_1 \cos(t+\theta_1) | dt = 4r r_1 = 4| \langle x, y \rangle_\CC |. \qedhere
\end{align*}
\end{proof}

We are now ready to prove:

\begin{proposition} \label{pro:zonoidmeasure}
Let $Z \subseteq \CC^n$ be a complex zonoid. Then there exists a Borel measure $\mu$ on the unit sphere $S \subseteq \CC^n$
such that 
$$h_Z(\theta)= \int_{S}\left|{\langle x,\theta \rangle}_\CC \right| d\mu(x).$$
\end{proposition}
It is possible to adapt the standard proof for real zonoids, but instead we will use the real result 
to derive the complex one. 
\begin{proof}

First note that $Z$ is also a real zonoid: Indeed, since every two-dimensional disc is a real zonoid it follows that 
every complex zonotope is a real zonoid. As $Z$ is the limit of such complex zonotopes it is also a real zonoid. Therefore 
there exists a Borel measure $\mu$ on $S$ such that 
$$h_Z(\theta)= \int_{S}\left|{\langle x,\theta \rangle}_\RR \right| d\mu(x).$$
However $Z$ is also a complex body, so $e^{ ti} Z = Z$ for all $t \in \RR$. Therefore 
$h_Z(e^{ti} \theta) = h_Z(\theta)$ for all $\theta \in S$ and $t \in \RR$. Averaging over $t$ we obtain 
\begin{align*}
h_Z(\theta) &= \frac{1}{2\pi} \int_0^{2\pi} h_Z(e^{ ti} \theta)dt 
= \frac{1}{2\pi} \int_0^{2\pi} \int_{S}\left|{\langle x,e^{ ti}\theta \rangle}_\RR \right| d\mu(x) \\
&= \frac{1}{2\pi} \int_{S}\int_0^{2\pi} \left|{\langle x,e^{ti}\theta \rangle}_\RR \right| d\mu(x)
= \frac{4}{2\pi} \int_S \left|{\langle x,\theta \rangle}_\CC \right| d\mu(x).
\end{align*}
This is the claimed result up to multiplying $\mu$ by a constant.
\end{proof} 

We finish this section by remarking that, as in the real case, not every complex convex body is a complex zonoid. 
For example, the unit ball of $\ell_1^n$ is not a zonoid for $n\ge3$. The proof is again out of the scope of this paper and
may appear elsewhere.

\subsection{Illumination of complex  zonoids}
\label{sec:zonoid_illumination}
Our next goal is to prove  Theorem \ref{thm:zonoid-illumination}, the resolution of Conjectures \ref{conj:complex_ill_conj} and \ref{conj:frac_complex_ill_conj} for complex zonoids. In the real case, Hadwiger's conjecture for zonoids was resolved by Boltjanski and Soltan in \cite{boltjanski92}, by reducing the problem to zonotopes.  
One can check that their reduction argument also works in the complex case, up to some minor adjustments. However, the proof of
\cite{boltjanski92} relies on Definition \ref{itm:abs-cont-zonoid} of Theorem \ref{the:mainzonoid} which is not 
well-known and hardly used in the literature. We shall give an alternative proof which is based on a similar  but simpler idea which relies on the well-known Definition  \ref{itm:measure-zonoid}. We will present the proof for complex zonoids, 
although it is applicable in the real case as well.

Before we present the proof, we need the technical fact that $\Ill$ and $\Ill^\ast$ are both upper semi-continuous:
\begin{proposition} \label{prop:usc}
    Fix a convex body $K \subseteq \RR^n$. Then: 
    \begin{enumerate}
    \item \label{enu:usc-classical} There exists $\delta > 0$ such that if $d_H(K,L) < \delta$ then $\ill(L) \le \ill(K)$. 
    \item \label{enu:usc-frac} For every $\epsilon > 0 $ there exists $\delta > 0$ such that if $d_H(K,L) < \delta$ then 
    $\ill^\ast (L) \le \ill^\ast(K) + \epsilon$.  
    \end{enumerate}
\end{proposition}

Part \ref{enu:usc-classical} of the proposition is well-known -- see Theorem 34.9 of \cite{Boltyanski1997}.  
However, as far as we know part \ref{enu:usc-frac} of the theorem is new, and fairly delicate, so we now present its
proof:
\begin{proof}
It is enough to show that for every $\epsilon>0$ and every measure
$\mu$ on $\Sph^{n-1}$ which illuminates $K$, there exists $\delta>0$
such that if $d_H(K,L)<\delta$ then the measure $\frac{1}{1-\epsilon}\mu$
illuminates $L$. 

For every compact set $C\subseteq\Sph^{n-1}$ and every $t>0$ we define
\[
U_{C,t}=\left\{ x\in\RR^{n}:\ x+tC\subseteq\Int(K)\right\} 
\]
 which is an open set in $\RR^{n}$. We claim that 
\[
\partial K\subseteq\bigcup\left\{ U_{C,t}\ :\ \begin{array}{l}
C\subseteq\Sph^{n-1}\text{ is compact with}\\
\mu(C)\ge1-\epsilon\text{ and }t>0
\end{array}\right\} .
\]
Indeed, for every $x\in\partial K$ we know that $A_{K}(x)=\left\{ v\in\Sph^{n-1}\ :\ v\text{ illuminates }x\right\} $
satisfies $\mu\left(A_{K}(x)\right)\ge1$. Since by definition 
\[
A_{K}(x)=\bigcup_{\ell=1}^{\infty}A_{K,\frac{1}{\ell}}(x):=\bigcup_{\ell=1}^{\infty}\left\{ v\in\Sph^{n-1}\ :\ x+\frac{1}{\ell}v\in\Int(K)\right\} ,
\]
 there exists an $\ell\in\NN$ such that $\mu\left(A_{K,\frac{1}{\ell}}(x)\right)\ge1-\frac{\epsilon}{2}$.
Since $\mu$ is a finite Borel measure on a compact metric space it
is inner regular, so there exists a compact set $C\subseteq A_{K,\frac{1}{\ell}}(x)$
such that $\mu(C)\ge1-\epsilon$. Then 
\[
x+\frac{1}{\ell}C\subseteq x+\frac{1}{\ell}A_{K,\frac{1}{\ell}}(x)\subseteq\Int(K),
\]
 and so $x\in U_{C,\frac{1}{\ell}}$ which proves the claim.

Since $\partial K$ is compact we can choose finitely many sets $U_{C_{1},t_{1}},U_{C_{2},t_{2}},\ldots,U_{C_{N},t_{N}}$
which cover $\partial K$. Writing $t=\min\left(t_{1},t_{2},\ldots,t_{N}\right)$,
it follows that $U_{C_{1},t},U_{C_{2},t},\ldots,U_{C_{N},t}$ also
cover $\partial K$. Choose compact subsets $A_{i}\subseteq U_{C_{i},t}$
such that we still have $\partial K\subseteq\bigcup_{i=1}^{N}A_{i}$.
We have $A_{i}+tC_{i}\subseteq U_{C_{i},t}+tC_{i}\subseteq\Int(K)$,
and since $A_{i}+tC_{i}$ is compact it follows that 
\[
\delta:=\frac{1}{4}d\left(\partial K,\bigcup_{i=1}^{N}\left(A_{i}+tC_{i}\right)\right)>0.
\]

Now we are finally ready to fix a convex body $L$ such that $d_{H}(K,L)<\delta$, and prove that $L$ is 
illuminated by $\frac{1}{1-\epsilon}\mu$. 

Indeed, fix $y\in\partial T$ and choose $x\in\partial K$ with $\left|x-y\right|<\delta$.
We know that $x\in A_{i}$ for some $1\le i\le N$. It follows that
\[
y+tC_{i}+\delta B_{2}^{n}\subseteq\left(x+\delta B_{2}^{n}\right)+tC_{i}+\delta B_{2}^{n}\subseteq A_{i}+tC_{i}+2\delta B_{2}^{n}\subseteq\Int(K),
\]
where we used the fact that $2\delta<d\left(A_{i}+tC_{i},\partial K\right)$.
However, since $d_{H}\left(K,T\right)<\delta$ every $x$ that satisfies
$x+\delta B_{2}^{n}\subseteq\Int(K)$ also satisfies $x\in\Int(T)$.
Hence $y+tC_{i}\subseteq\Int(T)$, so in particular the set of directions
$C_{i}$ illuminates $y$. Since 
\[
\left(\frac{1}{1-\epsilon}\mu\right)(C_{i})\ge\frac{1-\epsilon}{1-\epsilon}=1,
\]
it follows that $\frac{1}{1-\epsilon}\mu$ indeed illuminates $T$,
finishing the proof. 
\end{proof}

We are finally ready to prove Theorem \ref{thm:zonoid-illumination}: 
\begin{proof}[Proof of Theorem \ref{thm:zonoid-illumination}]
Assume $Z$ is a zonoid which is not a linear image of the polydisc. We will construct a zonotope $K$ which is also
not a linear image of the polydisc with the following property: For every $\delta > 0$ there exists a decomposition  $Z = Z_1 + Z_2$
such that $d_H(\frac{1}{m} Z_1, K) < \delta$ for some $m > 0$. 

We claim that finding such a zonotope $K$ completes our proof. To see this,  fix $\epsilon > 0$. By Proposition 
\ref{prop:usc} there exists $\delta > 0$ such that if $d_H(K,L) < \delta$ then 
$\Ill^\ast(L) \le \Ill^\ast(K) + \epsilon$. Choosing the decomposition $Z = Z_1 + Z_2$ that corresponds to this $\delta$ and 
using Fact \ref{sum_convex} we see that 
\[
\ill^\ast (Z) \le \ill^\ast (Z_1) = \ill^\ast \left(\frac{1}{m} Z_1\right) \le \ill^\ast(K) + \epsilon. 
\]
As this holds for all $\epsilon >0$, Theorem \ref{zonotopefrac} implies that 
$\ill^\ast (Z) \le \ill^\ast(K) < \ill^\ast(D)$. Using the same argument verbatim for classical illumination, we 
also obtain $\ill (Z) \le \ill(K) < \ill(D)$.

To find $K$ we use Proposition \ref{pro:zonoidmeasure} and write 
$$h_Z(\theta)= \int_{S}\left|{\langle x,\theta \rangle} \right| d\mu(x)$$ 
for a Borel measure $\mu$ on $S$. Since $Z$ is full-dimensional, the support of $\mu$ must contain a basis 
$\{a_1,a_2,\ldots,a_n\}$ of $\CC^n$. Since $Z$ is not a linear image of $D^n$, there exists $a_{n+1} \in \supp \mu$
which is not of the form $e^{ti} a_j$ for $t \in \RR$ and $1\le j \le n$. Therefore the zonotope 
$K = K_\CC (a_1 ,a_2, \ldots, a_{n+1})$ is full-dimensional and is not a linear image of the polydisc. 

Given $\delta > 0$, choose disjoint open neighborhoods $\{U_j\}_{j=1}^{n+1} \subseteq S$ of $\{a_j\}_{j=1}^{n+1}$ such that
all $U_j$'s have diameter smaller than $\frac{\delta}{n+1}$. Define $m = \min_j \mu(U_j)$, and consider the measures 
$\mu^{(j)} = \frac{m}{\mu(U_j)} \mu|_{U_j}$ and $\mu_1 = \sum_{j=1}^{n+1} \mu^{(j)}$. Since $\mu_1 \le \mu$, the zonoid
$Z_1$ with support function 
\[ h_{Z_1}(\theta) = \int_{S}\left|{\langle x,\theta \rangle} \right| d\mu_1(x)\]
is indeed a summand of $Z$. Moreover, for every $\theta \in S$ we have
\begin{align*}
\left| h_{\frac{1}{m} Z_1}(\theta) - h_K(\theta) \right| &= 
    \left| \frac{1}{m} \int_{S}\left|{\langle x,\theta \rangle} \right| d\mu_1(x) - 
    \sum_{j=1}^{n+1} \left| \langle a_j, \theta \rangle \right| \right| \\
&=  \left|  \sum_{j=1}^{n+1} \frac{1}{\mu(U_j)} \int_{U_j}\left|{\langle x,\theta \rangle} \right| d\mu(x) - 
    \sum_{j=1}^{n+1} \frac{1}{\mu(U_j)} \int_{U_j}\left| \langle a_j, \theta \rangle \right| d\mu(x) \right| \\
&\le \sum_{j=1}^{n+1}  \frac{1}{\mu(U_j)} \int_{U_j} \Bigl| \left|{\langle x,\theta \rangle} \right|
    -  \left|{\langle a_j,\theta \rangle} \right| \Bigr| d\mu(x).
\end{align*}
Since $x, a_j \in U_j$ and $U_j$  has diameter smaller than $\frac{\delta}{n+1}$ it follows that 
\[ 
\Bigl| \left|{\langle x,\theta \rangle} \right|  -  \left|{\langle a_j,\theta \rangle} \right| \Bigr|
\le \left| \langle x - a_j, \theta \rangle \right| \le |x-a_j| |\theta| < \frac{\delta}{n+1},
\]
and so 
\[ 
\left| h_{\frac{1}{m} Z_1}(\theta) - h_K(\theta) \right| < \sum_{j=1}^{n+1} \frac{\delta}{n+1} = \delta.
\]
It follows that $d_H(\frac{1}{m}Z_1, K) < \delta$ as claimed, and the proof is complete. 
\end{proof}

\section{A variant of the illumination problem}  \label{othervalues}

\subsection{Illuminating the polydisc with finite light sources}

As mentioned in the introduction, the illumination problem has an equivalent formulation in terms of "light sources".
We say that a light source $v \in \RR^n \setminus K$ illuminates $x \in \partial K$ if $x + t(x-v) \in \Int(K)$ for some
$t > 0$. We can then ask about the minimal number of light sources required to illuminate $\partial K$. 
The directions of illumination in the classical illumination problem can be interpreted as light sources "at infinity",
and it's therefore not hard to see that $\ill(K)$ is equal to the minimal number of light sources needed to 
illuminated $K$. 

We can now formulate a variant of the problem, where the light sources are not allowed to escape to infinity. Instead
we fix $r > 1$, and define $\ill_r(K)$ to be the minimal number of light sources $v_1, v_2, \ldots, v_N \in rK$ which illuminate $K$. Similarly, we define
\[
\ill_{r}^{\ast}(K)=\inf\left\{ \mu(rK):\ \begin{array}{l}
\mu\text{ is a Borel measure on }rK\text{ and for all }x\in\partial K\\
\text{we have }\mu\left(\left\{ v:\ x\text{ is illuminated by the light source }v\right\} \right)\ge1
\end{array}\right\}. 
\]
We will now study this more restricted problem in the case where $K = D^n$. 

A variant of Lemma  \ref{lem:extreme} can be established for illumination by light sources, so it is again enough
to illuminate the torus $D_0^n = \extreme(D^n)$. It also remains true that a light source $v=(v_1,v_2,\ldots,v_n)$ 
illuminates $x = (x_1, x_2,\ldots, x_n) \in D_0^n$ if and only if every $v_i$ illuminates $x_i \in \partial D$. 
Clearly every $v_i$ illuminates a larger arc of $\partial D$ as $|v_i|$ increases, so we may assume that $|v_i|=r$
for all $1\le i \le n$, i.e. that all our light sources belong to $r D_0^n$.

Let $v = (r e^{2\pi  \alpha_1 i}, r e^{2\pi  \alpha_2 i}, \ldots, r e^{2\pi  \alpha_n i}) \in rD_0^n$ be such a 
light source. It is a simple exercise in trigonometry that $v_j = r e^{2\pi  \alpha_j i}$ illuminates a point 
$x_j = e^{2\pi  \beta_j i} \in \partial D$ if and only if $d(\alpha_j, \beta_j) < \frac{1}{2\pi}\arccos\frac{1}{r}$,
where $d(\alpha_j, \beta_j)$ denotes the usual (symmetric) distance on $\TT$ -- see Figure \ref{fig:illuminatingball}.
\begin{figure}[ht!]
    \centering
    \includegraphics[width=0.5\linewidth]{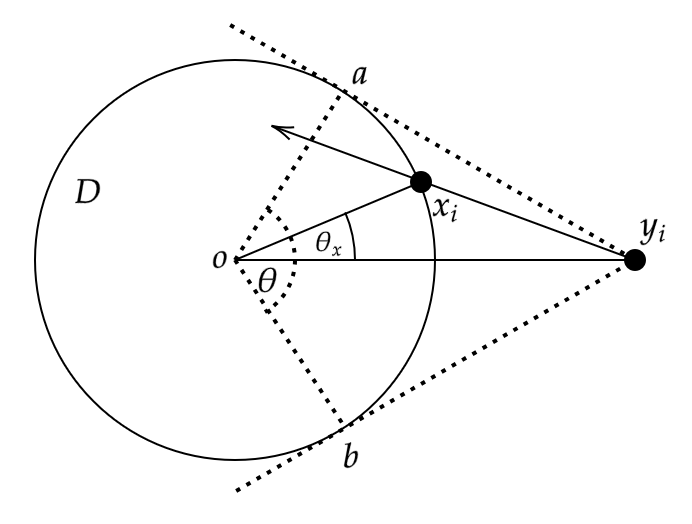}
    \caption{$y_i$ will illuminate $x_i$ if $x_i$ is in the arc $ab$, with $a,b$ being the points of tangency corresponding to $y_i$. We have $|\arg y_i-\arg x_i|=\theta_x<\theta$,which can be calculated using $|oy_i|=r$.}
    \label{fig:illuminatingball}
\end{figure}
Therefore the set $Y_v \subset D_0^n$ of points illuminated by the light source $v$ is precisely 
\begin{equation*}
\label{eq:open_cube_extra}
Y_v=\left\{(e^{2\pi\beta_1 i},\cdots,e^{2\pi\beta_n i})\in D_0^n\ :\  d(\alpha_j,\beta_j)<\frac{1}{2\pi}\arccos{\frac{1}{r}} \text{ for all } 1\le j \le n \right\}.
\end{equation*}
Under our usual identification of $D_0^n$ with $\TT^n$, the set $Y_v$ is a cube of side length $\epsilon_r = \frac{1}{\pi} \arccos \frac{1}{r}$. 
We therefore obtain an extension of Proposition \ref{prop:torus_polydisc_ill}:
\begin{proposition} \label{prop:cover-extended}
We have $\ill_r(D^n)=N(\TT^n,(0,\epsilon_r)^n)$ and $\ill_r^*(D^n)=N^*(\TT^n,(0,\epsilon_r)^n)$ for all $r>1$, 
where $\epsilon_r = \frac{1}{\pi} \arccos \frac{1}{r}$.
\end{proposition}

Note that $\epsilon_r \to \frac{1}{2}$ as $r\to \infty$, recovering Proposition \ref{prop:torus_polydisc_ill}.
Also, since $\epsilon_2 = \frac{1}{3}$, we can use Propositions \ref{prop:lowerboundm} and \ref{prop:upperboundm} 
to precisely compute e.g. $\ill_2(D^n) = \frac{3^{n+1}-1}{2}$. 

\subsection{Covering the torus with cubes of arbitrary side length }
From Propositions \ref{prop:cover-extended} and \ref{prop:fractionalcover} it follows immediately that 
\[ 
\ill_r^\ast(D^n) = N^\ast\left(\TT^n, (0,\epsilon_r)^n \right) = 
\left( \frac{\pi}{\arccos\frac{1}{r}} \right)^n 
\]
for all $r>1$. In contrast, in order to compute $\ill_r(D^n)$ for a general $r>1$ one needs to know the covering
numbers $N(\TT^n, (0,\epsilon)^n)$ for a general $\epsilon > 0$. This is an open problem of independent interest,
which we will not completely solve here. 
However, as already mentioned, some results about covering with \emph{closed} cubes appeared in \cite{Bogdanov22} and
\cite{MCELIECE1973119}, and we now briefly study what these results imply for covering with open cubes. 

To relate covering with open cube to covering with closed cubes we use the following lemma:
\begin{lemma} \label{closecovertorus}
For every $0<\epsilon<1$ we have 
\[ 
N\left(\TT^n, (0,\epsilon)^n\right) = \lim_{\eta \to \epsilon^-} 
N\left(\TT^n, [0,\eta]^n\right). \]
More explicitly, there exists $\delta > 0$ such that if $\epsilon - \delta < \eta < \epsilon$ then 
$N\left(\TT^n, (0,\epsilon)^n\right) = N\left(\TT^n, [0,\eta]^n\right) $.
\end{lemma}
\begin{proof}
Write $N = \left(\TT^n, (0,\epsilon)^n\right)$ and choose open cubes $C_1,C_2,\ldots,C_N$ of side-length $\epsilon$ that cover $\TT^n$.  Writing $C_i = x_i + (0,\epsilon)^n$ and  $C_i(\eta) = x_i + (0,\eta)^n$ for all $0 < \eta < \epsilon$, we  have 
\[ 
\bigcup_{i=1}^{N} \bigcup_{\eta < \epsilon} C_i(\eta) = \bigcup_{i=1}^N C_i  = \TT^n.
\]
Therefore, by compactness, there exists a finite sub-cover $C_{i_1}(\eta_1), C_{i_2}(\eta_2), \ldots, C_{i_k}(\eta_k)$ of $\TT^n$. Since 
$C_i(\eta) \subseteq C_i(\eta')$ for $\eta < \eta'$ we may assume that every index $1\le i\le N$ appears only once in our sub-cover, i.e., our sub-cover has the form $C_1(\eta_1), C_2(\eta_2), \ldots, C_N(\eta_N)$. Setting $\eta_0 = \max\{\eta_1,\eta_2,\ldots,\eta_n\} < \epsilon$, we see that $N\left(\TT^n, (0,\eta_0)^n\right) \le N$. Since the opposite inequality is trivial, it follows that for all $\eta_0 < \eta < \epsilon$ we have $N\left(\TT^n, [0,\eta]^n\right) = N$ as claimed. 
\end{proof}
In \cite[Theorem 2]{MCELIECE1973119}, the following lower bound was given for $N\left(\TT^n, [0,\epsilon]^n\right)$: 
Define a sequence $\{b_n(\epsilon)\}_{n=1}^\infty$ inductively by $b_0(\epsilon) = 1$ and 
$b_{n+1}(\epsilon) = \left\lceil\frac{b_n(\epsilon)}{\epsilon}\right\rceil $. Then 
$N\left(\TT^n, [0,\epsilon]^n\right) \ge b_n(\epsilon)$ for all $\epsilon > 0$. Using this bound we prove:
\begin{proposition} \label{lowerboundG}
    Define a sequence $\{a_n(\epsilon)\}_{n=1}^\infty$ inductively by setting $a_0(\epsilon) = 1$ and 
$a_{n+1}(\epsilon) = \left\lfloor\frac{a_n(\epsilon)}{\epsilon}\right\rfloor + 1$. Then 
$ N\left(\TT^n, (0,\epsilon)^n\right) \ge a_n(\epsilon)$ for all $\epsilon > 0$. 
\end{proposition}
\begin{proof}
We prove by induction on $n$ that $\lim_{\eta \to \epsilon^{-}} b_n(\eta) = a_n(\epsilon)$. 
Indeed, assume this holds for a given $n$. This means that there exists $\delta>0$ such that 
$b_n(\eta) = a_n(\epsilon)$ for all $\epsilon - \delta < \eta < \epsilon$. Therefore 
\[ 
\lim_{\eta \to \epsilon^{-}} b_{n+1}(\eta) = 
\lim_{\eta \to \epsilon^{-}} \left\lceil\frac{b_{n}(\eta)}{\eta}\right\rceil = 
\lim_{\eta \to \epsilon^{-}} \left\lceil\frac{a_{n}(\epsilon)}{\eta}\right\rceil = 
\left\lfloor\frac{a_{n}(\epsilon)}{\epsilon}\right\rfloor + 1 = a_{n+1}(\epsilon),
\]
where we used the fact that $\lim_{y \to x^+} \lceil y \rceil = \lfloor x \rfloor + 1$. This finishes the inductive proof.

Using Lemma \ref{closecovertorus}, the result of \cite{MCELIECE1973119} and the claim above we see that indeed
\[ 
N\left(\TT^n, (0,\epsilon)^n\right) = \lim_{\eta \to \epsilon^-} N\left(\TT^n, [0,\eta]^n\right)
\ge \lim_{\eta \to \epsilon^-} b_n(\eta) = a_n(\epsilon). \qedhere
\]
\end{proof}

Note that for $m \in \NN$ we have $a_n\left(\frac{1}{m}\right) = \frac{m^{n+1}-1}{m-1}$, 
so Proposition \ref{lowerboundG} is an extension of Proposition \ref{prop:lowerboundm}. In this case we proved a 
matching upper bound in Proposition \ref{prop:upperboundm}, but in general the lower bound given by
Proposition \ref{lowerboundG} does not have to be sharp. 

Better results can be given in low dimensions. In two dimensions, a complete answer was given in 
\cite{MCELIECE1973119}: For every $\epsilon >0$ we have 
$N(\TT^2,[0,\epsilon]^2)=\left\lceil\frac{1}{\epsilon}\left\lceil\frac{1}{\epsilon}\right\rceil\right\rceil$.
Using Lemma \ref{closecovertorus} and the identity $\lim_{y \to x^+} \lceil y \rceil = \lfloor x \rfloor + 1$
we conclude that 
\[ 
N(\TT^2, (0,\epsilon)^2) = 
\left\lfloor{ \frac{1}{\epsilon} } \left(\left\lfloor{ \frac{1}{\epsilon} }\right\rfloor + 1\right) \right\rfloor + 1.
\]
In three dimensions, the covering numbers $N(\TT^3,[0,\epsilon]^3)$ were computed in \cite{Bogdanov22} 
for some values of $\epsilon$.  Combining Theorems 1,2 and 5 of \cite{Bogdanov22}  with Lemma \ref{closecovertorus} we deduce that
$$
N(\TT^3,(0,\epsilon)^3)=
        \begin{cases}
			4, & \epsilon\in\left( \frac{3}{4},1\right]\\
            5, & \epsilon\in\left( \frac{2}{3},\frac{3}{4}\right]\\
            7, & \epsilon\in\left( \frac{3}{5},\frac{2}{3}\right]\\
            8, & \epsilon\in\left( \frac{1}{2},\frac{3}{5}\right]\\
            m^3+m^2+m+1, & \epsilon\in\left( \frac{1}{m+\frac{1}{m^2+m+1}},\frac{1}{m}\right],m\in\NN\\
            m^3, & \epsilon\in\left( \frac{1}{m},\frac{1}{m-\frac{1}{m^2-1}}\right],m\in\NN.
		\end{cases}
$$
To the best of our knowledge no such results are known in dimension $n>3$.

\bibliographystyle{amsplain_Abr_NoDash}
\bibliography{citations}
\end{document}